\pdfoutput=1

\documentclass[12pt,a4paper]{amsart}

\usepackage[utf8]{inputenc}
\usepackage[T1]{fontenc}
\usepackage{lmodern}
\usepackage[kerning, spacing, tracking]{microtype}

\usepackage[style=numams, 
            sorting=nyt, 
            isbn=false,  
            backref=true, 
            backend=biber]{biblatex}

\usepackage{amsmath,amssymb, amsthm}
\usepackage{mathtools}  
\usepackage{enumitem}
\usepackage{tikz-cd}
\usepackage{stmaryrd}  

\usepackage[colorlinks, linkcolor=blue, citecolor=blue, pdfa]{hyperref}

\usepackage[hcentering, hscale=0.72, vscale=0.75 ]{geometry}

\newtheorem{introthm}{Theorem}

\newtheorem{introprop}[introthm]{Proposition}

\swapnumbers
\newtheorem{thm}{Theorem}[section]
\newtheorem{cor}[thm]{Corollary}
\newtheorem{lemma}[thm]{Lemma}
\newtheorem{prop}[thm]{Proposition}
\newtheorem{conj}[thm]{Conjecture}

\theoremstyle{definition}
\newtheorem{defi}[thm]{Definition}
\newtheorem{example}[thm]{Example}

\theoremstyle{remark}
\newtheorem{remark}[thm]{Remark}

\newlist{enums}{enumerate}{2}  
\setlist[enums,1]{label=\textup{(\alph*)}}
\setlist[enums,2]{label=\textup{(\roman*)}}

\renewcommand{\phi}{\varphi}
\renewcommand{\theta}{\vartheta}
\newcommand{\eps}{\varepsilon}

\renewcommand{\leq}{\leqslant}

\newcommand{\nbd}{\nobreakdash-\hspace{0pt}}  

\newcommand{\defemph}[1]{\textbf{#1}}

\newcommand{\rats}{\mathbb{Q}}   
\newcommand{\compl}{\mathbb{C}}  

\newcommand{\crp}[1]{\mathbb{#1}}     

\newcommand{\sym}[1]{\mathrm{S}_{#1}}

\DeclarePairedDelimiterX{\brcls}[3]{\llbracket}{\rrbracket}{#1,#2,#3}
\DeclarePairedDelimiter{\algcls}{\lbrack}{\rbrack}

\newcommand{\iso}{\cong}    
\newcommand{\nteq}{\trianglelefteq} 
\newcommand{\into}{\hookrightarrow}  
\newcommand{\tensor}{\otimes}
\newcommand{\bigtensor}{\bigotimes}

\newcommand{\oG}{\widehat{G}}
\newcommand{\oH}{\widehat{H}}

\DeclareMathOperator*{\bigdcup}{\mathaccent\cdot{\mathop{\bigcup}}}

\DeclarePairedDelimiter{\abs}{\lvert}{\rvert}
\DeclarePairedDelimiter{\erz}{\langle}{\rangle}

\DeclareMathOperator{\Ker}{Ker}        
\DeclareMathOperator{\Aut}{Aut}
\DeclareMathOperator{\Z}{\mathbf{Z}}
\DeclareMathOperator{\C}{\mathbf{C}}         
\DeclareMathOperator{\Gal}{Gal}  
\DeclareMathOperator{\Irr}{Irr}
\DeclareMathOperator{\Lin}{Lin}
\DeclareMathOperator{\mat}{\mathbf{M}}      
\DeclareMathOperator{\enmo}{End} 
\DeclareMathOperator{\Hom}{Hom}
\DeclareMathOperator{\Tr}{Tr}
\DeclareMathOperator{\SC}{\mathcal{SC}}
\DeclareMathOperator{\Br}{Br}
\DeclareMathOperator{\BrCliff}{BrCliff}
\DeclareMathOperator{\Cores}{Cores}
\DeclareMathOperator{\Res}{Res}

\hypersetup{pdfauthor={Frieder Ladisch},
            pdftitle={The Schur-Clifford subgroup
                      of the Brauer-Clifford group},
            pdfkeywords={Brauer-Clifford group, Clifford theory,
                         Character theory of finite groups,
                         Brauer group,
                         Schur subgroup
                        }
            }

\begin{document}
\title[The Schur-Clifford subgroup]{The Schur-Clifford subgroup\\%
                           of the Brauer-Clifford group}
\author{Frieder Ladisch}
\address{Universität Rostock\\
         Institut für Mathematik\\
         Ulmenstr.~69, Haus~3\\
         18057 Rostock\\
         Germany}
\email{frieder.ladisch@uni-rostock.de}
\thanks{Author partially supported by the DFG (Project: SCHU 1503/6-1)}
\subjclass[2010]{Primary 20C15, Secondary 16K50}
\keywords{Brauer-Clifford group, Clifford theory,
          Character theory of finite groups, 
          Brauer group, Schur subgroup}

\begin{abstract}
  We define 
  a Schur-Clifford subgroup of Turull's Brauer-Clifford group,
  similar to the Schur subgroup of the Brauer group. 
  The Schur-Clifford subgroup 
  contains exactly the equivalence classes coming from the intended 
  application to Clifford theory of finite groups.
  We show that the Schur-Clifford subgroup is indeed a subgroup 
  of the Brauer-Clifford group, 
  as are certain naturally defined subsets.
  We also show that this Schur-Clifford subgroup behaves well 
  with respect to restriction and corestriction maps between 
  Brauer-Clifford groups.
\end{abstract}
\maketitle

\section{Introduction}
\label{sec:intro}
The Brauer-Clifford group has been introduced by 
Turull~\cite{turull09,turull11}
to handle Clifford theory in finite groups in a way that 
respects fields of values and other rationality properties 
of the characters involved.
%
%
Let $\kappa\colon \oG\to G$ be a surjective homomorphism
of finite groups with kernel $N\nteq \oG$. 
In other words, 
  \[ 
    \begin{tikzcd}
       1 \rar & N \rar[hook] & \oG \rar{\kappa} & G \rar & 1
    \end{tikzcd}
  \]
  is an exact sequence of groups.
Let $\theta\in \Irr N$, where $\Irr N$ denotes, as usual, the set of 
irreducible complex characters of the group $N$.
For convenient reference, we call the pair
$(\theta, \kappa)$ a
\defemph{Clifford pair}.
Note that $N$, $\oG$ and $G$ are determined by $(\theta,\kappa)$
as the kernel, the domain and the image 
of $\kappa$, respectively.
We will often consider different Clifford pairs which have
the group $G$ in common, while $\oG$ and $N$ vary.
To describe this conveniently,
we also say that $(\theta,\kappa)$ is a 
\defemph{Clifford pair over $G$} in the above situation.

Let $\crp{F}\subseteq \compl$ be a field.
Turull has shown how to associate a commutative
$\crp{F}$\nbd algebra 
$\Z(\theta,\kappa,\crp{F})$, on which $G$ acts, 
to a Clifford pair $(\theta,\kappa)$ and 
the field $\crp{F}$~\cite[Definition~7.1]{turull09}.
This is called the \defemph{center algebra} of
$(\theta,\kappa)$ over $\crp{F}$.
Then the Brauer-Clifford group
\[ \BrCliff( G, \Z(\theta,\kappa, \crp{F}))
\]
is defined
(We will review its precise definition below).
Moreover, Turull showed~\cite[Definition~7.7]{turull09} 
how to associate an element
\[\brcls{\theta}{\kappa}{\crp{F}}\in \BrCliff(G,\Z(\theta,\kappa,\crp{F}))
\]
to $(\theta,\kappa)$ and $\crp{F}$.
We call this the
\defemph{Brauer-Clifford class} 
of the pair $(\theta,\kappa)$ over $\crp{F}$.
This element determines in some sense the character theory 
of $\oG$ over $\theta$,
including fields of values and Schur indices over the field
$\crp{F}$~\cite[Theorem~7.12]{turull09}.

In this paper, we consider the subset of $\BrCliff(G,Z)$
that consists of elements that come from this construction.
To make this more precise, we need the fact that 
the Brauer-Clifford group is a covariant functor
in the second variable.
This means that if $\alpha\colon R \to S$ 
is a homomorphism of $G$-rings, then there is
a corresponding homomorphism of abelian groups
\[ \overline{\alpha}= \BrCliff(G,\alpha)\colon   
   \BrCliff(G,R)\to \BrCliff(G,S).
\]
If $\alpha$ happens to be an isomorphism, then
$\overline{\alpha}$ is an isomorphism, too.

For a fixed commutative $G$\nbd algebra $Z$ over the field $\crp{F}$,
we define
$ \SC^{(\crp{F})}(G,Z)$
as the set of all
$ \overline{\alpha}\brcls{\theta}{\kappa}{\crp{F}} $
such that $(\theta,\kappa)$ is a Clifford pair over $G$
and $\alpha\colon \Z(\theta,\kappa,\crp{F})\to Z$
is an isomorphism of $G$\nbd algebras.
In short,
\[ \SC^{(\crp{F})}(G,Z)
    := \{ \overline{\alpha}
          \brcls{\theta}{\kappa}{\crp{F}}
          \mid \Z(\theta,\kappa,\crp{F}) \iso Z
               \text{ via } \alpha 
      \}.
\]

Of course, for the wrong $Z$, it may happen that
$\SC^{(\crp{F})}(G,Z)$ is empty.
If not,
we call $\SC^{(\crp{F})}(G,Z)$ the
\defemph{Schur-Clifford subgroup} (over $\crp{F}$)
of the Brauer-Clifford group $\BrCliff(G,Z)$.
The main result of this paper justifies this name:
\begin{introthm}\label{t:main}
  If $Z\iso \Z(\theta,\kappa,\crp{F})$ for 
  some Clifford pair $(\theta,\kappa)$,
  then $\SC^{(\crp{F})}(G,Z)$ is a subgroup of
  $\BrCliff(G,Z)$.
\end{introthm}
With respect to the Brauer-Clifford group,
the Schur-Clifford subgroup is the same
as is the Schur subgroup~\cite{Yamada74}
with respect to the Brauer group.

We will usually write $\SC(G,Z)$ instead of
$\SC^{(\crp{F})}(G,Z)$. 
This is justified by the following result:
\begin{introprop}\label{p:fieldindep}
  Let $Z$ be a $G$\nbd ring and assume
  that $Z\iso \Z(\theta,\kappa,\crp{F})$
  for some Clifford pair $(\theta,\kappa)$
  and some field\/ $\crp{F} \subseteq \compl $.
  Then $Z^G$, the centralizer of $G$ in $Z$,
  is a field.
  There is a unique subfield\/
  $\crp{K}\subseteq \compl$ with $\crp{F} \subseteq \crp{K}$
  such that $Z^G \iso \crp{K}$ as algebra over\/ $\crp{F}$.
  For this field $\crp{K}$, we have
  \[ \SC^{(\crp{F})}(G,Z)= \SC^{(\crp{K})}(G,Z).
  \]  
\end{introprop}
This means that we can always work with the field
$\crp{K}\iso Z^G$ which is uniquely determined by 
the $G$\nbd algebra $Z$.

We also consider certain subclasses $\mathcal{X}$
of the class of all Clifford pairs. 
Possible examples are 
the class $\mathcal{X}$ of all pairs $(\theta,\kappa)$
such that $N=\Ker \kappa$ is abelian (nilpotent, solvable \dots),
or the class of all pairs such that
$\Ker \kappa$ is a central subgroup of the domain of 
$\kappa$,
or the class of all pairs such that $\theta$ is induced
from a linear character of some subgroup, and so on.
For any class $\mathcal{X}$,
we define a \defemph{restricted Schur-Clifford subset}
\[ \SC_{\mathcal{X}}(G,Z)
   = 
   \{ \overline{\alpha}
             \brcls{\theta}{\kappa}{\crp{F}}
             \mid \Z(\theta,\kappa,\crp{F}) \iso Z
                  \text{ via } \alpha 
                  \text{ for some }
                  (\theta,\kappa)\in \mathcal{X}
   \}.
\]
Of course, depending on $\mathcal{X}$ and $Z$, this may not
be a subgroup, even if
$\SC(G,Z)\neq \emptyset$.
Our proof of Theorem~\ref{t:main} will show, however,
that $\SC_{\mathcal{X}}(G,Z)$ \emph{is} a subgroup
for certain classes $\mathcal{X}$. 
(See Theorem~\ref{t:scresgroup} for the case 
 where $Z$ is a field.)
Finally, we mention the following: Assume that
$Z\iso\Z(\theta,\kappa,\crp{F})$ for some Clifford pair
$(\theta,\kappa)$ over $G$.
Then it is known that
$Z \iso \crp{F}(\theta)\times \dotsm \times \crp{F}(\theta)$
is a direct product of isomorphic fields
and that $G$ permutes transitively the factors of $Z$.
Let $e\in Z$ be a primitive idempotent and let
$H\leq G$ be its stabilizer in $G$.
Thus $K=Ze\iso \crp{F}(\theta)$ is an $H$\nbd algebra.
Turull~\cite[Theorem~5.3]{turull09} has shown that
$\BrCliff(G,Z)\iso \BrCliff(H, K) $ canonically.
\begin{introthm}\label{t:reductofields}
  In the situation just described,
  the canonical isomorphism 
  \[\BrCliff(G,Z)\iso \BrCliff(H,K)\]
  restricts to an isomorphism
  \[\SC(G,Z)\iso \SC(H,K).\]
\end{introthm}
We mention that it is not particularly difficult to see that the map
$\BrCliff(G,Z)\to \BrCliff(H,K)$ maps
$\SC(G,Z)$ into $\SC(H,K)$
(replace $(\theta,\kappa)$ by 
 $(\theta, \pi)$, where
 $\pi$ is the restriction of $\kappa$ to
 $\kappa^{-1}(H)$).
The proof that every class in $\SC(H,K)$ comes from a class
in $\SC(G,Z)$ is more complicated and involves a construction
using wreath products.

The proofs are organized as follows:
First, we prove Theorem~\ref{t:main} and Proposition~\ref{p:fieldindep} 
in the case where $Z$ is a field.
(This is done in Sections~\ref{sec:proof} and~\ref{sec:depfield}.
In Sections~\ref{sec:brauercliff}--\ref{sec:semiinv}
we review the necessary definitions and background we need,
and define the Schur-Clifford group.)
Then, in Section~\ref{sec:notfield},
 we prove Theorem~\ref{t:reductofields} in the sense that we show
that the set $\SC(G,Z)$ and the subgroup
$\SC(H,K)$ correspond under the canonical isomorphism
of Brauer-Clifford groups.
This then implies that 
Theorem~\ref{t:main} and Proposition~\ref{p:fieldindep}
are true for arbitrary $Z$.
In Section~\ref{sec:res}, we show that the Schur-Clifford group
behaves well with respect to restriction of groups
(this yields the easy part of Theorem~\ref{t:reductofields}).
In the more technical Section~\ref{sec:cores}, we consider the relation
of the Schur-Clifford group to the \emph{corestriction} map
as defined in our paper~\cite{ladisch15a}.
This will be needed in forthcoming work, but is included here since 
the proof is similar to the proof of Theorem~\ref{t:reductofields},
and we can use some lemmas from this proof.
Finally, in the last section, we describe the Schur-Clifford group
$\SC(G,\crp{F})$, when $G$ acts trivially on $\crp{F}$,
in terms of the Schur subgroup and a certain subgroup of 
$H^2 (G, \crp{F}^*)$,
which has been studied by Dade~\cite{dade74}.
Based on this special case, we propose a conjecture.

\section{Review of the Brauer-Clifford group}
\label{sec:brauercliff}
The Brauer-Clifford group was first defined by Turull~\cite{turull09} 
for commutative simple $G$\nbd algebras.  
Later, he gave a different, but equivalent definition~\cite{turull11}.
Herman and Mitra~\cite{HermanMitra11} have shown how to extend
the definition to arbitrary commutative $G$\nbd rings, 
and that the Brauer-Clifford group is the same
as the equivariant Brauer group defined earlier
by Fröhlich and Wall~\cite{frohlichwall00}. 

Let $G$ be a group.
A 
\defemph{$G$\nbd ring} is a ring
$R$ on which $G$ acts by ring automorphisms.
\emph{In this paper, ``ring'' always means
``ring with one'' and 
``ring homomorphism'' always means
``unital ring homomorphism''.}
We use exponential notation $r\mapsto r^g$
($r\in R$, $g\in G$)
to denote the action of $G$
on $R$.
A \defemph{$G$\nbd ring homomorphism}
(or \defemph{homomorphism of $G$\nbd rings})
is a ring homomorphism $\phi\colon R\to S$
between $G$\nbd rings
such that $r^g \phi = (r\phi)^g$
for all $r\in R$ and $g\in G$.

Let $R$ be a commutative $G$\nbd ring. 
A \defemph{$G$\nbd algebra $A$ over $R$}
is a $G$\nbd ring $A$, together with an
homomorphism of $G$\nbd rings
$\eps\colon R \to \Z(A)$.
In particular, $A$ is an algebra over $R$.
We usually suppress $\eps$ and simply write
$ra$ for $(r\eps)\cdot a$. 
A $G$\nbd algebra $A$ over the commutative $G$\nbd ring $R$
is called an \defemph{Azumaya $G$\nbd algebra over $R$}
if it is Azumaya over $R$ as an $R$\nbd algebra.
We will only need the cases where $R$ is a field
or a direct product of fields.
If $R$ is a field, 
$A$ is Azumaya over $R$ iff $A$ is a finite dimensional 
central simple $R$\nbd algebra.
If $R$ is a direct product of fields, then
$A$ is Azumaya iff the algebra unit induces an isomorphism
$R\iso \Z(A)$, and   for every primitive idempotent $e$ of $R$,
the algebra $Ae$ is central simple over the field $Re$.

If $A$ and $B$ are Azumaya $G$\nbd algebras over the 
$G$\nbd ring $R$, then
$A \tensor_{R} B$ is an Azumaya $G$\nbd algebra over $R$.

The Brauer-Clifford group $\BrCliff(G,R)$
consists of equivalence classes of Azumaya 
$G$\nbd algebras over $R$.
To define the equivalence relation,
we need some definitions first.
Let $R$ be a $G$\nbd ring.
The \defemph{(skew) group algebra} $RG$ 
is the set of formal sums
\[ \sum_{g\in G} g r_g, \quad r_g\in R,
\]
which form a free $R$\nbd module with basis $G$,
and multiplication defined by $gr \cdot hs = gh r^hs$
and extended distributively.
Let $R$ be commutative and $P$ a right $RG$\nbd module.
Then $\enmo_R(P)$ is a $G$\nbd algebra over $R$.
If $P_R$ is an $R$-progenerator, then 
$\enmo_R(P)$ is Azumaya over $R$.
(Note that any finitely generated 
 non-zero $R$\nbd module is an $R$\nbd progenerator
 if $R$ is a field.)
By definition,
a \defemph{trivial Azumaya $G$\nbd algebra} over $R$
is an algebra of the form $\enmo_R(P)$,
where $P$ is a module over the group algebra
$RG$ such that $P_R$ is a progenerator over $R$.

Two Azumaya $G$-algebras $A$ and $B$ over the $G$\nbd ring $R$
are called \defemph{equivalent}, 
if there are trivial Azumaya $G$\nbd algebras
$E_1$ and $E_2$ such that
\[ A\tensor_R E_1 \iso B\tensor_R E_2.
\]
In other words, $A$ and $B$ are equivalent, if there
are $RG$\nbd modules $P_1$ and $P_2$ 
which are progenerators over $R$ and
such that
\[ A \tensor_R \enmo_R(P_1) \iso
   B \tensor_R \enmo_R(P_2).
\]

The \defemph{Brauer-Clifford group} $\BrCliff(G,R)$
is the set of equivalence classes of 
Azumaya $G$\nbd algebras over $R$
together with the (well defined) multiplication induced by 
the tensor product $\tensor_R$.
With this multiplication,
$\BrCliff(G,R)$ becomes an abelian torsion 
group~\cites[Theorem~6.2]{turull11}[Theorem~5]{HermanMitra11}.

\section{The Schur-Clifford group}
\label{sec:schurcliff}
Let $(\theta,\kappa)$ be a Clifford pair, 
as defined in the introduction.
Recall that this means that there is an exact sequence
of groups
\[ \begin{tikzcd}
     1 \rar & N \rar[hook]   & \oG \rar{\kappa} & G \rar & 1,
   \end{tikzcd}
\]
and that $\theta\in \Irr N$.

In this situation, $G\iso \oG/N$ acts on the 
set $\Irr N$ of irreducible
characters of $N$.
Moreover, $G$ acts on $\Z(\compl N)$, 
the center of the group algebra.
Recall that with every $\theta\in \Irr N$ there is associated
a central primitive idempotent
\[ e_{\theta} = \frac{\theta(1)}{\abs{N}}
                 \sum_{n\in N} \theta(n^{-1})n.
\]
The actions of $G$ on $\Irr N $ and $\Z(\compl N)$
are compatible in the sense that
$(e_{\theta})^g = e_{\theta^g}$.

Let $\crp{F}\leq \compl$ be a subfield of $\compl$.
Let $\Gamma:=\Gal(\compl/\crp{F})$ be the group of all
field automorphisms of $\compl$ that leave every element
of $\crp{F}$ fixed.
(For our purposes, $\compl$ can be replaced by any 
 algebraically closed field containing $\crp{F}$,
 or by a Galois extension of $\crp{F}$
 that is a splitting field of $\oG$.)
Then $\Gamma$ acts on $\Irr N$ 
by $\theta^{\alpha}(n)=\theta(n)^{\alpha}$
and on
$\compl N$ by acting on coefficients.
Again, these actions are compatible in the sense that
$(e_{\theta})^{\alpha}=e_{\theta^{\alpha}}$.
Observe, however, that if we view $\theta$ as a function defined
on $\compl N$, then
$\theta^{\alpha}(\sum_n a_n n) 
  = \sum_n a_n \theta(n)^{\alpha}
  = \theta(\sum_n a_n^{\alpha^{-1}} n)^{\alpha}$.
  
Since the actions of $\Gamma$ and $G$ commute,
we get an action of $\Gamma \times G$ on 
$\Irr N$ respective $\Z(\compl N)$.

We write
  \[ e_{\theta, \crp{F}}
      = \sum_{\alpha\in \Gal(\crp{F}(\theta)/\crp{F})}
        (e_{\theta})^{\alpha}.
  \]
Thus 
$e_{\theta,\crp{F}}$ is the central primitive idempotent
of $\crp{F}N$ associated to $\theta$,
and $\crp{F}N e_{\theta,\crp{F}}$ the corresponding simple
summand. 
It is well known that 
$\Z(\crp{F}Ne_{\theta,\crp{F}})\iso \crp{F}(\theta)$:
an isomorphism is given by restricting the
\defemph{central character}
\[ \omega_{\theta}\colon \Z(\crp{F}N)\to \crp{F}(\theta),
   \quad \omega_{\theta}(z)= \frac{\theta(z)}{\theta(1)}
\] 
to
$\Z(\crp{F}Ne_{\theta,\crp{F}})$~\cite[Proposition~38.15]{huppCT}.

We write
  $(e_{\theta,\crp{F}})^{\oG}$ or $(e_{\theta,\crp{F}})^G$ for the sum
  of the different $G$\nbd conjugates of $e_{\theta,\crp{F}}$.
Then $(e_{\theta, \crp{F}})^G$ is a primitive idempotent
in 
$(\Z(\crp{F}N))^G = \C_{\Z(\crp{F}N)}(G) 
                  =\crp{F}N\cap\Z(\crp{F}\oG)$.
It is the sum of the idempotents corresponding to irreducible
characters in the $(\Gamma \times G)$-orbit of $\theta$.
Following Turull~\cite{turull09}, we define
\[ \Z(\theta, \kappa, \crp{F})
   = \Z(\crp{F}N)(e_{\theta,\crp{F}})^G. 
\]
This is called 
the \defemph{center algebra} of $(\theta,\kappa)$ over $\crp{F}$.
It is a commutative $G$-algebra over $\crp{F}$,
and a direct sum of the algebras
$\Z(\crp{F}N(e_{\theta,\crp{F}})^g)\iso \crp{F}(\theta)$.
These summands are permuted transitively by the group $G$.
Thus 
\[\Z(\theta,\kappa,\crp{F})\iso 
   \crp{F}(\theta)\times \dotsm \times \crp{F}(\theta)
\]
is simple in the sense that it has no nontrivial
$G$\nbd invariant ideals.
The center algebra depends only
on the $\Gamma \times G$\nbd orbit of $\theta$,
where $\Gamma= \Gal(\compl/\crp{F})$.

Turull~\cite{turull09} has also shown how to
associate an equivalence class of simple
$G$\nbd algebras with center 
$\Z(\theta,\kappa,\crp{F})$
to the triple $(\theta,\kappa, \crp{F})$.
Namely, let $V$ be any non-zero
$\crp{F}\oG$\nbd module such that
$V=V(e_{\theta,\crp{F}})^G$.
Such a module is called
\defemph{$\theta$-quasihomogeneous}.
Let $S=\enmo_{\crp{F}N}(V)$.
Then $G\iso \oG/N$ acts on $S$.
For $z\in Z = \Z(\theta,\kappa,\crp{F})$, 
the map $v\mapsto vz$ is in $S$.
This defines an isomorphism of $G$\nbd algebras
$\Z(\theta,\kappa,\crp{F})\iso \Z(S)$.
The equivalence class of $S$ in the 
Brauer-Clifford group $\BrCliff(G,Z)$
depends only on $(\theta,\kappa,\crp{F})$,
not on the choice of the $\crp{F}\oG$\nbd module 
$V$~\cite[Theorem~7.6]{turull09}.
We write $\brcls{\theta}{\kappa}{\crp{F}}$ to denote
this element of the Brauer-Clifford group,
and call it the 
\defemph{Brauer-Clifford class} of $(\theta,\kappa)$ over $\crp{F}$.
\begin{defi}
  Let $S$ be a $G$\nbd algebra, with $\crp{F}\subseteq \Z(S)^G$.
   We call $S$ a 
   \defemph{Schur $G$\nbd algebra} over $\crp{F}$,   
   if there is a Clifford pair $(\theta,\kappa\colon \oG\to G)$
   and a $\theta$-quasihomogeneous
   $\crp{F}\oG$-module $V$, such that
  $S\iso \enmo_{\crp{F}N} V$ as $G$\nbd algebras over $\crp{F}$.  
\end{defi}
This means in particular that $S\in \brcls{\theta}{\kappa}{\crp{F}}$.
This definition is not completely analogous to 
the definition used in the theory of the Schur subgroup
of the Brauer group~\cite[cf.][]{Yamada74}.
Recall that a central simple algebra $S$ over a field 
$K$ is called a Schur algebra if 
$S$ is isomorphic to a direct summand of the group algebra $ K N$
of some finite group $N$.
In the situation above,
$V_{\crp{F}N}\iso \left(\sum_i U_i\right)^n$, 
where the $U_i$ are simple $\crp{F}N$\nbd modules 
which are conjugate in $\oG$.
Thus $S\iso \prod_i \mat_n(D)$, where
$D\iso \enmo_{\crp{F}N}(U_i)$.
We also have
$\crp{F}Ne_{\theta,\crp{F}}\iso \mat_k(D)$
and
$\crp{F}N (e_{\theta,\crp{F}})^G
 \iso \prod_i \mat_k(D)$
for some $k$.
Thus $S$ and $\crp{F}N (e_{\theta,\crp{F}})^G$
are Morita equivalent.

We mention that one can show that in fact every $G$\nbd algebra
in a Brauer-Clifford class 
$\brcls{\theta}{\kappa}{\crp{F}}$ 
is a Schur $G$\nbd algebra over $\crp{F}$,
if the class contains one Schur algebra.

\begin{lemma}[{Turull~\cite[Proposition~7.2]{turull09}}]
  Let $S$ be a Schur $G$-algebra over $\crp{F}$. Then
  $\Z(S)\iso \crp{E}\times\dots \times \crp{E}$, where
  $\crp{E}$ is a field contained in a cyclotomic extension of
  $\crp{F}$.
\end{lemma}
\begin{proof}
  When $S\in \brcls{\theta}{\kappa}{\crp{F}}$,
  then $\Z(S)\iso \Z(\theta,\kappa,\crp{F})$,
  and $\Z(\theta,\kappa,\crp{F})$ is a direct product
  of fields of the form 
  $\Z(\crp{F}N)e_{\theta^g,\crp{F}}
    \iso\crp{F}(\theta^g)
    =\crp{F}(\theta)$.  
\end{proof}

\begin{defi}
  Let $Z$ be a commutative simple $G$-algebra,
  and assume $Z\iso \Z(\theta,\kappa,\crp{F})$
  for some Clifford pair $(\theta,\kappa)$. 
  The \defemph{Schur-Clifford group} of
  $Z$ is the set of all equivalence classes
  of central simple
  $G$\nbd algebras over $Z$ that contain a Schur
  $G$\nbd algebra over $\crp{F}$.
  We denote it by $\SC^{(\crp{F})}(G,Z)$
  or simply $\SC(G,Z)$.
  In other words,
  \[ \SC(G,Z)=\{ \overline{\alpha}
                 \brcls{\theta}{\kappa}{\crp{F}}\mid
                 \Z(\theta,\kappa,\crp{F})\iso Z 
                 \text{ via } \alpha\},
  \]
  where 
  \[ \overline{\alpha}\colon 
       \BrCliff(G,\Z(\theta,\kappa,\crp{F})) 
       \to \BrCliff(G,Z)
  \]
  is induced by $\alpha$.
\end{defi}
Omitting the field $\crp{F}$ from the notation 
will be justified once we have proved
Proposition~\ref{p:fieldindep} from the introduction.

Our aim is to show that 
the Schur-Clifford group $\SC(G,Z)$ is indeed a subgroup of the
Brauer-Clifford group.

It will be convenient to have a definition of 
a restricted Schur-Clifford group, 
where not all pairs $(\theta,\kappa)$ are allowed.
For example, we could consider only pairs
where the kernel of $\kappa$ is 
contained in the center of $\oG$. 
\begin{defi}
  Let $G$ be a finite group and $Z$ a commutative
  simple $G$\nbd algebra over $\crp{F}$.
  Let $\mathcal{X}$ be a class of 
  Clifford pairs. 
  Then
  \[ \SC_{\mathcal{X}}(G,Z)
    =\{ \overline{\alpha}\brcls{\theta}{\kappa}{\crp{F}}
        \mid (\theta,\kappa)\in \mathcal{X}
        \quad\text{and}\quad
        \Z(\theta,\kappa,\crp{F})\stackrel{\alpha}{\iso} Z \}.
  \]
\end{defi}
Of course, depending on $\mathcal{X}$,
$\SC_{\mathcal{X}}(G,Z)$ might not be a subgroup
of the Brauer-Clifford group.
\begin{example}\label{ex:complcohom}
  Let $Z=\compl $ with trivial $G$-action.
  Then $\BrCliff(G,\compl)\iso H^2(G,\compl^*)$.
  Every simple $G$-algebra $S$ over $\compl $
  is isomorphic to a matrix ring
  $\mat_n(\compl)$ 
  and defines a projective representation
  $G\to S^*$.
  By the classical theory of Schur, 
  the projective representation can be lifted to an
  ordinary representation
  $\oG\to S^*$, where $\oG$ has a cyclic central subgroup
  $C$ such that $\oG/C\iso G$.
  Thus 
  $\BrCliff(G,\compl)=\SC(G,\compl)=\SC_{\mathcal{C}}(G,\compl)$,
  where $\mathcal{C}$ 
  denotes the class of pairs $(\theta,\kappa)$, 
  where $\Ker\kappa$ is a central cyclic subgroup.
  
  It is also known that there is a central extension
  \[ \begin{tikzcd}
      1 \rar & H^2(G,\compl^*) \rar & \oG \rar{\kappa} & G \rar &1
     \end{tikzcd}
  \]
  such that every projective representation of $G$ can be lifted 
  to an ordinary representation of $\oG$.
  The group $\oG$ is then called a Schur representation group.
  Thus for 
  $\mathcal{C} = \mathcal{C}_{\kappa}
               = \{ (\theta,\kappa) \mid 
                    \theta\in \Lin(\Ker\kappa)
                 \}$,
  where $\kappa\colon \oG\to G$ is a fixed
  Schur representation group of $G$,
  we also have $\BrCliff(G,\compl)=\SC_{\mathcal{C}}(G,\compl)$.
\end{example}

\section{Semi-invariant Clifford pairs}
\label{sec:semiinv}
Let $(\theta,\kappa)$ be a Clifford pair and 
$\crp{F}\subseteq \compl$ a field
such that $\Z(\theta,\kappa,\crp{F})$ is a field.
By the remarks above on the center algebra, this means
that $(e_{\theta,\crp{F}})^G= e_{\theta,\crp{F}}$
is invariant in $G$ or $\oG$
(where, as before, $\kappa\colon\oG\to G$).
In this case, we say that
$\theta$ is 
\defemph{semi-invariant} in $\oG$ over $\crp{F}$.
We also say that $(\theta,\kappa)$ is 
a semi-invariant Clifford pair over $\crp{F}$.
A character $\theta$ is semi-invariant over $\crp{F}$ in $\oG$
if and only if for every
$g\in \oG$ there is 
$\alpha_g\in \Gal(\crp{F}(\theta)/\crp{F})$ such that
$\theta^{g\alpha_g}=\theta$.
The map $g\mapsto \alpha_g$ is a group homomorphism
from $\oG$ to $\Gal(\crp{F}(\theta)/\crp{F})$
with kernel $\oG_{\theta}$~\cite[Lemma~2.1]{i81b}.
Since $N=\Ker\kappa\subseteq \oG_{\theta}$,
the map $g\mapsto\alpha_g$ defines an action
of $G$ on $\crp{F}(\theta)$.
Thus
$\crp{F}(\theta)$ can be viewed as a
$G$\nbd ring.
Note that a Galois conjugate $\theta^{\sigma}$
of $\theta$ yields the same homomorphism, 
since $\Gal(\rats(\theta)/\rats)$ is abelian.

The central character 
$\omega_{\theta}$ defines an isomorphism
$\Z(\theta,\kappa,\crp{F})\iso \crp{F}(\theta)$.
From
\[ \theta(z^g)= \theta^{g^{-1}}(z)= \theta(z)^{\alpha_g}
\]
we see that 
$\omega_{\theta}\colon \Z(\theta,\kappa,\crp{F})\to \crp{F}(\theta)$
is an isomorphism of $G$\nbd algebras.
In fact, any other $\crp{F}$\nbd isomorphism
$\Z(\theta,\kappa,\crp{F})\to \crp{F}(\theta)$
is an isomorphism of $G$\nbd algebras, 
since $\Gal(\crp{F}(\theta)/\crp{F})$ is abelian.

When considering the Schur-Clifford group of a field
$Z$ with $G$\nbd action, it is thus natural to assume that
$Z=\crp{E}$ is a subfield of $\compl$
(or any algebraically closed field of characteristic zero
 fixed in advance and 
 in which all characters are assumed to take values).
Thus
\[ \SC(G,\crp{E})
    =
    \{ \overline{\omega_{\theta}}\brcls{\theta}{\kappa}{\crp{F}}
       \mid \Z(\theta,\kappa,\crp{F})\stackrel{\omega_{\theta}}{\iso} 
       \crp{E}= \crp{F}(\theta)
    \}.
\]
Let us mention here that the notation
$\SC(G,\crp{E})$, 
as well as 
$\BrCliff(G,R)$, 
does not reflect the dependence on the actual
action of $G$ on $R$. 
A ring may be equipped with different actions of the same group $G$.
In particular, two different
semi-invariant Clifford pairs
$(\theta_1,\kappa_1)$ and $(\theta_2,\kappa_2)$
with
$\crp{F}(\theta_1)=\crp{F}(\theta_2)=\crp{E}$ 
may yield elements of different 
Schur-Clifford groups
which are both denoted by $\SC(G,\crp{E})$.
This is due to the fact that it can happen that 
$\Z(\theta_1,\kappa_1,\crp{F})$ and
$\Z(\theta_2,\kappa_2,\crp{F})$
are not isomorphic as $G$\nbd algebras, 
although they are isomorphic fields.
In the equation displayed above, we assume that
an action of $G$ on $\crp{E}$ is fixed and that
$\omega_{\theta}$ yields an isomorphism
of $G$\nbd algebras
from $\Z(\theta,\kappa,\crp{F})$ to $\crp{E}$.

\section{Proof of the subgroup properties}
\label{sec:proof}
We work now to show that 
$\SC_{\mathcal{X}}(G,Z)$ is a subgroup of 
$\BrCliff(G,Z)$, 
given some mild conditions on $\mathcal{X}$ and $Z$.
We do this first for the case where $Z=\crp{E}$ is a field.
The general case will follow once we have proved
Theorem~\ref{t:reductofields} from the introduction.

Note that when 
$\compl \supseteq\crp{E}\iso \Z(\theta,\kappa,\crp{F})$
for some Clifford pair $(\theta,\kappa)$,
then necessarily $\crp{E}=\crp{F}(\theta)$
is contained in $\crp{F}(\eps)$
for some root of unity $\eps$.
The next lemma yields the converse.
\begin{lemma}\label{l:scid}
  Let $G$ be a group which acts on the field\/
  $\crp{E}\subseteq \compl$.
  Assume that\/ $\crp{F}$ is a subfield of the fixed
  field\/ $\crp{E}^G$.
  If there is a root of unity $\eps$
  such that\/ $\crp{E}\subseteq \crp{F}(\eps)$,
  then\/ $\crp{E}$ is a Schur $G$-algebra over\/ $\crp{F}$.
\end{lemma}
\begin{proof}
   Set $\Gamma = \Gal(\crp{F}(\eps)/\crp{F})$.
   As $\Gamma$ is abelian, every field between 
   $\crp{F}$ and $\crp{F}(\eps)$ is normal over $\crp{F}$.
   In particular, $\Gamma$ acts on $\crp{E}$. 
   Since $\crp{E}$ is a $G$\nbd algebra by
   assumption, the group $G$ acts on $\crp{E}$, too. 
   We let $U$ be the pullback in 
   \[ \begin{tikzcd}
           U \rar \dar & \Gamma \dar \\
           G \rar & \Aut(\crp{E}) \rlap{,}
     \end{tikzcd}
   \]
   that is,
   \[ U = \{(g,\sigma)\in G\times \Gamma
              \mid x^g=x^{\sigma}
            \text{ for all } x\in \crp{E}\}.
   \]
   Let $C=\erz{\eps}$. 
   The group $U$ acts on $C$ by
   $c^{(g,\sigma)}=c^\sigma$.
   Let  $\oG=UC$ be the semidirect
   product of $U$ and $C$ with respect to this action.
   The map $\kappa\colon \oG \to G$ sending
   $(g,\sigma)c$ to $g$ is an epimorphism.
     (The surjectivity follows from the 
      well known fact that every 
      $\crp{F}$\nbd automorphism of $\crp{E}$
      extends to one of $\crp{F}(\eps)$.)
   Let $N$ be the kernel of $\kappa $.
   Then $N=AC$, where
   \[ A=\{(1,\sigma)\mid x^{\sigma}=x
   \text{ for all } 
   x\in \crp{E}\} \iso \Gal(\crp{F}(\eps)/\crp{E}).
   \]
   Set $V=\crp{F}(\eps)$ and define a $\oG$-module structure 
   on $V$ by
   \[ v(g,\sigma)c = v^{\sigma} c
      \quad\text{for $v\in \crp{F}(\eps)$,
                 $(g,\sigma)\in U$ and  $c\in C$.} 
   \]
   
   We claim that $\crp{E}\iso \enmo_{\crp{F}N}(V)$ as
   $G$\nbd algebras.
   For $x\in \crp{E}$, 
   right multiplication
   $\rho_x\colon v\mapsto vx$ ($v\in V$)
   is an $\crp{F}N$\nbd endomorphism.
   The map $x\mapsto \rho_x$
   defines an injection
   of $\crp{F}$\nbd algebras
   $\crp{E}\to \enmo_{\crp{F}N}(V)$.   
   Since $\crp{F}C$ acts as $\crp{F}(\eps)$ on $V$ from the
   right, $\enmo_{\crp{F}N}(V)$ can be identified with a
   subfield of $\crp{F}(\eps)$ (via the above map $x\mapsto \rho_x$), 
   and since $A$ acts as
   $\Gal(\crp{F}(\eps)/\crp{E})$ on $V$, this subfield is
   just $\crp{E}$.

   It remains to show that
   $x\mapsto \rho_x$ commutes with the action of $G$.
   Remember that for $g\in G$, the endomorphism 
   $(\rho_x)^g$ is defined by 
   $v\mapsto v\hat{g}^{-1}\rho_x \hat{g}$, 
   where $\hat{g}\in \oG$ is some element
   with $\hat{g}\kappa =g$.
   Here we can choose 
   $\hat{g}=(g,\sigma)$ with suitable $\sigma\in \Gamma$.      
   We get that
   \[ v(\rho_x)^g= v(g^{-1},\sigma^{-1})x(g,\sigma)
        =  v^{\sigma^{-1}}x(g,\sigma)
        = vx^{\sigma} = v x^g ,\]
   where the last equation follows from the fact that
   $(g,\sigma)\in U$ and thus $x^{\sigma}=x^g$.
   So $(\rho_x)^g = \rho_{x^g}$,
   and the proof is finished.
\end{proof}
From the proof of the last lemma, we get the following 
more precise result:
\begin{cor}
    Assume the hypotheses of Lemma~\ref{l:scid}.
    Then $\algcls{\crp{E}}=\brcls{\theta}{\kappa}{\crp{F}}$,
    where $(\theta,\kappa)$ has the following properties:
    \begin{enums}
    \item $N:=\Ker\kappa = AC$ is the semidirect product
          of a cyclic group $C\iso\erz{\eps}$ 
          and a subgroup $A$ of $\Aut(C)$
          such that $A\iso \Gal(\crp{F}(\eps)/\crp{E})$.
    \item $\theta = \lambda^N$, where $\lambda $ is a 
          faithful linear character of $C$.
    \end{enums}
\end{cor}
\begin{proof}
  Let $C=\erz{\eps}$, 
  $\oG$ and $V$ be as in the proof of Lemma~\ref{l:scid}.
  View the embedding
   $\lambda\colon C\into \crp{F}(\eps)$ 
  as linear character of
  $C$.
  The $\crp{F}C$-module $V=\crp{F}(\eps)$ 
  affords the 
  character 
  $\Tr_{\crp{F}}^{\crp{F}(\eps)}(\lambda)
     = \sum_{\sigma\in \Gamma} \lambda^{\sigma}$,
  where $\Gamma= \Gal(\crp{F}(\eps)/\crp{F})$.
  Since $\theta=\lambda^N\in \Irr N$,
  the character of $V_N$ contains $\theta$ as irreducible constituent.
  This proves the corollary.
\end{proof}
In the next result, we only assume that
$Z$ is a commutative $G$\nbd ring and that
$\crp{F}\subseteq Z^G$ is a field.
\begin{lemma}\label{l:scinv}
  If $S$ is a Schur $G$-algebra over $\crp{F}$,
  then $S^{\text{op}}$ is a Schur $G$-algebra over 
  $\crp{F}$.
  In fact, when $\algcls{S}=\brcls{\theta}{\kappa}{\crp{F}}$,
  then $\algcls{S^{\text{op}}}= \brcls{\overline{\theta}}{\kappa}{\crp{F}}$.
\end{lemma}
\begin{proof}
  Suppose that
  $\oG/N \iso G$ and that $V$ is a (right) $\crp{F}\oG$\nbd module such
  that $S= \enmo_{\crp{F}N}(V)$.
  Let $V'= \Hom_{\crp{F}}(V,\crp{F})$. Then
  $V'$ becomes a right $\crp{F}\oG$\nbd module by
    $v (v'g) = (vg^{-1})v'$ for $v\in V$, $v'\in V'$.
    Define $\alpha\colon S^{\text{op}}\to \enmo_{\crp{F}N}(V')$
    by $v((v') s^{\alpha}) = (vs)v'$.
    Then $\alpha$ is an isomorphism of
    $G$\nbd algebras. In fact, with the same definition
    one gets an isomorphism from
    $\enmo_{\crp{F}}(V)^{\text{op}}$ onto
    $\enmo_{\crp{F}}(V')$ as $\oG$\nbd algebras, as is well known.
    The first part of the lemma follows from this fact.
    
    For the second, recall that when $V$ affords the character $\chi$, then
    $V'$ affords the character
    $g\mapsto \chi(g^{-1})=\overline{\chi(g)}$.
    Thus if $\theta$ is an irreducible constituent of the character
    ov $V_{N}$, then
    $\overline{\theta}$ is an irreducible constituent
    of the character of $(V')_{N}$.
    The proof is finished.
\end{proof}
Our next goal is to show that when
$a_1$ and $a_2\in \SC(G,Z)$, 
      then $a_1a_2\in \SC(G,Z)$.
For simplicity of notation, we assume again that
$Z=\crp{E}$ is a subfield of $\compl$.
Let $(\theta_i,\kappa_i)$ be two Clifford pairs
over the same group $G$
such that 
$\Z(\theta_1,\kappa_1,\crp{F})\iso \crp{E}
 \iso \Z(\theta_2,\kappa_2,\crp{F})$
as $G$\nbd algebras.
We want to construct a Clifford pair
$(\theta,\kappa)$ such that
\[ \brcls{\theta}{\kappa}{\crp{F}}
    = \brcls{\theta_1}{\kappa_1}{\crp{F}}
      \brcls{\theta_2}{\kappa_2}{\crp{F}}.
\]
For $i=1$, $2$ let
\[ \begin{tikzcd}
           1 \rar & N_i \rar & G_i \rar{\kappa_i} &G \rar & 1
   \end{tikzcd}
\]
be the exact sequence of groups
belonging to the Clifford pair $(\theta_i, \kappa_i)$.
Let $\oG=\{(g_1,g_2)\in G_1\times G_2
              \mid g_1\kappa_1=g_2\kappa_2\}$
be the pullback of the maps $\kappa_1$ and $\kappa_2$
and let $\kappa\colon \oG\to G$ be the composition
$\oG\to G_i\to G$:
\[ \begin{tikzcd}
       \oG \dar \rar \drar[color=red]{\kappa}
           & G_1 \dar{\kappa_1} 
        \\
       G_2 \rar{\kappa_2} 
           & G
   \end{tikzcd}.
\]
We also write $\kappa_1 \times_{G} \kappa_2$ for this map.
($G_1\times_G G_2$
  is a common notation for the pullback $\oG$.)
Then $\kappa=\kappa_1\times_G \kappa_2$ is surjective,
since both $\kappa_1$ and $\kappa_2$ are,
and the kernel is $N_1\times N_2$.
Thus we have an exact sequence 
\[ \begin{tikzcd}
        1  \rar & N_1 \times N_2 \rar & \oG \rar{\kappa} & G \rar & 1.
   \end{tikzcd}
\]
We write $\theta_1 \times \theta_2$
for the irreducible character of $N_1\times N_2$
defined by $(\theta_1 \times\theta_2)(n_1,n_2)=\theta_1(n_1)\theta_2(n_2)$.
We have now defined a new Clifford pair
\[  (\theta_1\times\theta_2,\kappa_1\times_G \kappa_2)\]
over $G$.
\begin{lemma}\label{l:scprod}
    In the above situation, we have
    \[ \brcls{\theta_1\times\theta_2}{\kappa_1\times_G\kappa_2}{\crp{F}}
       = \brcls{\theta_1}{\kappa_1}{\crp{F}} 
         \brcls{\theta_2}{\kappa_2}{\crp{F}}
    \]
    in $\BrCliff(G,\crp{E})$.
\end{lemma}
\begin{proof}
  That $\Z(\theta_1,\kappa_1,\crp{F})\iso
        \Z(\theta_2,\kappa_2,\crp{F})\iso \crp{E}$
  as $G$\nbd algebras means that
  $\crp{F}(\theta_1)=\crp{F}(\theta_2)=\crp{E}$
  and that for every $g\in G$ there is 
  $\alpha_g\in \Gal(\crp{E}/\crp{F})$
  such that
  $\theta_1^{g\alpha_g}=\theta_1$ and
  $\theta_2^{g\alpha_g}=\theta_2$.
  Clearly, we also have
  $\crp{F}(\theta_1\times\theta_2)=\crp{E}$.
  Since 
  $(\theta_1\times\theta_2)^{g\alpha_g}
   = \theta_1^{g\alpha_g}\times \theta_2^{g\alpha_g}
   = \theta_1\times \theta_2$,
  we have that
  $\Z(\theta_1\times \theta_2,\kappa_1\times_G\kappa_2,\crp{F})
  \iso \crp{E}$ as $G$\nbd algebras.
      
  Suppose that
  $V_i $ are $\theta_i$\nbd quasihomogeneous
  $\oG_i$\nbd modules with
  $S_i = \enmo_{\crp{F}N_i}(V_i)$ as $G$\nbd algebras.
  We have to show that
  \[ S_1 \tensor_{\crp{E}} S_2 
     \in
        \brcls{\theta_1\times\theta_2}{\kappa_1\times_G\kappa_2}{\crp{F}}.
  \]
  
  Since $\crp{E}\iso \Z(\crp{F}N_i e_{\theta_i,\crp{F}})$
  for $i=1$, $2$, we can view $V_1$ and $V_2$ as vector spaces over
  $\crp{E}$.
  (The vector space structures depend on the concrete isomorphisms,
   which we can assume to be given by $\omega_{\theta_i}$,
   but in fact the arguments to follow work for any choices
   of isomorphisms
   $\crp{E}\iso \Z(\crp{F}N_i e_{\theta_i,\crp{F}})$.)
  
  Let $V= V_1\tensor_{\crp{E}}V_2$ and define an action of $\oG$
   on $V$ by
   $(v_1\tensor v_2)(g_1,g_2)=v_1g_1\tensor v_2g_2$.
   This is well defined: Let $z\in \crp{E}$. Then
   \begin{align*}
     (v_1z\tensor v_2 )(g_1,g_2)
       &= v_1z g_1 \tensor v_2 g_2
       \\
       &= v_1 g_1 z^{g_1\kappa_1}\tensor v_2g_2
       \\
       &= v_1 g_1 \tensor z^{g_1\kappa_1} (v_2g_2)
       \\
       &= v_1 g_1 \tensor (v_2g_2)z^{g_2\kappa_2}
            & &\text{since $g_1\kappa_1=g_2\kappa_2$}
       \\
       &= v_1 g_1 \tensor v_2 z g_2
       \\
        &= (v_1\tensor zv_2)(g_1, g_2).
   \end{align*}
   In this way, $V=V_1\tensor_{\crp{E}} V_2$ becomes an
   $\crp{F}\oG$\nbd module.
   
   As $\crp{E}N_i$\nbd module,
   the character of $V_i$ is a multiple of $\theta_i$.
   Thus $V$ as an $\crp{E}[N_1\times N_2]$ module has as 
   character a multiple of $\theta_1\times \theta_2$.
   It follows that $V$ as 
   $\crp{F}[N_1\times N_2]$\nbd module 
   is $(\theta_1 \times \theta_2)$\nbd quasihomogeneous.

   Since 
   $\crp{E}$ is isomorphic to the centers
   of the various algebras involved,
   we have
   \begin{align*}
     \enmo_{\crp{F}[N_1\times N_2]}(V)
       &= \enmo_{\crp{E}[N_1\times N_2]} (V_1 \tensor_{\crp{E}} V_2)
       \\
       &\iso \enmo_{\crp{E}N_1}(V_1) \tensor_{\crp{E}}
             \enmo_{\crp{E}N_2}(V_2)
       \\
       &= \enmo_{\crp{F}N_1}(V_1) \tensor_{\crp{E}}
                        \enmo_{\crp{F}N_2}(V_2)
       \\
       &= S_1 \tensor_{\crp{E}} S_2.
   \end{align*}
   The isomorphism sends
   $s_1\tensor s_2 \in S_1\tensor_{\crp{E}}S_2 $ 
   to the map
   $(v_1\tensor v_2) \mapsto (v_1s_1)\tensor (v_2s_2)$.
   It is easy to see that this is an isomorphism of
   $G$-algebras.
   Thus
   $S_1\tensor_{\crp{E}}S_2 \iso \enmo_{\crp{E}[N_1\times N_2]}(V)
    \in \brcls{\theta_1\times\theta_2}{\kappa_1\times \kappa_2}{\crp{F}}$
   as was to be shown.
\end{proof}
\begin{remark}
  The construction of the Clifford pair
  $(\theta_1\times \theta_2, \kappa_1\times_G \kappa_2)$
  does not depend on the fact that the center algebras
  of $(\theta_1,\kappa_1)$ and $(\theta_2,\kappa_2)$ are fields
  and isomorphic.
  Without this assumption, one can show the following:
  The center algebra 
  $Z:=\Z(\theta_1\times\theta_2,\kappa_1\times_G \kappa_2,\crp{F})$
  is a direct summand of 
  $Z_1\tensor_{\crp{F}}
   Z_2$, where
  $Z_i= \Z(\theta_i,\kappa_i,\crp{F})$.
  The isomorphism type of $Z$ depends on the choice
  of $\theta_i$ in its orbit under the action 
  of $G\times \Gal(\crp{F}(\theta_i)/\crp{F})$.
  This is different from the situation in
  Lemma~\ref{l:scprod}, where
  $\Z(\theta_1\times\theta_2^{\alpha}, \kappa_1\times\kappa_2,\crp{F})
   \iso \crp{E}$ as $G$\nbd algebras
  for all  $\alpha\in \Gal(\crp{E}/\crp{F})$.
  
  In the general situation,
  it follows that there are $G$\nbd algebra homomorphisms
  $Z_i \to Z_1\tensor_{\crp{F}}Z_2\to  Z$.
  If $S_i$ is a $G$\nbd algebra in 
  $\brcls{\theta_i}{\kappa_i}{\crp{F}}$, 
  then 
  $(S_1\tensor_{Z_1}Z)\tensor_Z (S_2\tensor_{Z_2} Z)$
  is in 
  $\brcls{\theta_1\times\theta_2}{\kappa_1\times_G \kappa_2}{\crp{F}}$.
  Since we do not need these more general results, 
  we do not pause to prove them.
\end{remark}
We are now ready to prove Theorem~\ref{t:main} for $Z=\crp{E}$ a field.
\begin{thm}\label{t:scgroup}
  Let\/ $\crp{E}$ be a field on which a group $G$ acts.
  Suppose $\crp{E}$ is contained in a cyclotomic extension
  of\/ $\crp{F}\subseteq\crp{E}^G$.
  Then the Schur-Clifford group $\SC(G,\crp{E})$
   is  a subgroup of
  the Brauer-Clifford group $\BrCliff(G,\crp{E})$.
\end{thm}
We prove this together with a theorem about restricted 
Schur-Clifford groups.
\begin{thm}\label{t:scresgroup}
  In the situation of Theorem~\ref{t:scgroup},
  let $\mathcal{X}$ be a class of Clifford pairs
  $(\theta,\kappa)$ (over $G$),
  such that the following conditions hold:
  \begin{enums}
  \item $\SC_{\mathcal{X}}(G,\crp{E})\neq \emptyset$,
  \item When $(\theta,\kappa)\in \mathcal{X}$,
        then $(\overline{\theta},\kappa)\in \mathcal{X}$,
  \item When $(\theta_1,\kappa_1)$ and 
        $(\theta_2,\kappa_2)\in \mathcal{X}$,
        then $(\theta_1\times \theta_2,\kappa_1\times_G\kappa_2)
        \in \mathcal{X}$.
  \end{enums}  
  Then $\SC_{\mathcal{X}}(G,\crp{E})$ 
  is a subgroup of $\BrCliff(G,\crp{E})$.
\end{thm}
\begin{proof}[Proof of Theorem~\ref{t:scgroup} and~\ref{t:scresgroup}]
  Recall that
  \[\SC_{\mathcal{X}}(G,\crp{E})
      =\{ \brcls{\theta}{\kappa}{\crp{F}}
          \mid (\theta,\kappa)\in \mathcal{X}
          \quad\text{and}\quad
          \Z(\theta,\kappa,\crp{F})\iso \crp{E} \}.
  \]
  Lemma~\ref{l:scid} yields that
  $\algcls{\crp{E}}\in \SC(G,\crp{E})$.
  Lemma~\ref{l:scinv} yields that when
  $\brcls{\theta}{\kappa}{\crp{F}}\in \SC_{\mathcal{X}}(G,\crp{E})$,
  then 
  \[\brcls{\theta}{\kappa}{\crp{F}}^{-1}
    = 
    \brcls{\overline{\theta}}{\kappa}{\crp{F}}
    \in \SC_{\mathcal{X}}(G,\crp{E}).
  \]
  Finally,
  when $\brcls{\theta_1}{\kappa_1}{\crp{F}}$
  and $\brcls{\theta_2}{\kappa_2}{\crp{F}}
      \in \SC_{\mathcal{X}}(G,\crp{E})$,
  then $\Z(\theta_1,\kappa_1,\crp{F})
       \iso \crp{E}
       \iso \Z(\theta_2,\kappa_2,\crp{F})$
  as $G$\nbd algebras,
  and Lemma~\ref{l:scprod}
  yields that 
  \[ \brcls{\theta_1}{\kappa_1}{\crp{F}}
    \brcls{\theta_2}{\kappa_2}{\crp{F}}
    = 
    \brcls{\theta_1\times\theta_2}{\kappa_1\times_G\kappa_2}{\crp{F}}
   \in \SC_{\mathcal{X}} (G,\crp{E}).
  \qedhere
  \]
\end{proof}
The conditions in Theorem~\ref{t:scresgroup}
for $\mathcal{X}$ are not necessary for 
$\SC_{\mathcal{X}}(G,\crp{E})$ to form a subgroup.
The following is an example:
\begin{cor}\label{c:cyclickernelclass}
  Let $\mathcal{C}$ be the class of semi-invariant
  Clifford pairs $(\theta,\kappa)$ such that
  $N=\Ker\kappa$ is cyclic.
  Suppose that $\crp{E}=\crp{F}(\eps)$
  for some root of unity $\eps$.
  Then $\SC_{\mathcal{C}}(G,\crp{E})$
  is a subgroup of
  $\BrCliff(G,\crp{E})$.
\end{cor}
\begin{proof}
  We begin by showing 
  $[\crp{E}]\in \SC_{\mathcal{C}}(G,\crp{E})$.
  The action of $G$ on $\crp{E}=\crp{F}(\eps)$
  yields an action of $G$ on $C=\erz{\eps}$.
  Let $\oG=GC$ be the semidirect product
  and $\kappa\colon \oG\to G$ the corresponding epimorphism
  with kernel $C$.
  View the inclusion $C\subset \crp{E}$ as a linear character
  $\lambda$.
  Make $V=\crp{E}$ into a $\oG$\nbd module
  by defining $v(gc)= v^g c$.
  Then $V$ is $\lambda$\nbd quasihomogeneous and
  $\enmo_{\crp{F}C}(V)= \enmo_{\crp{F}(\eps)}(V)\iso \crp{E}$
  as $G$\nbd algebras.
  Thus $[\crp{E}]= \brcls{\lambda}{\kappa}{\crp{F}}
        \in \SC_{\mathcal{C}}(G,\crp{E})$.
        
  Now let $\mathcal{A}$ be the class of semi-invariant 
  Clifford pairs over $G$ such that
  $\Ker\kappa$ is abelian.
  By Theorem~\ref{t:scresgroup} and the preceding paragraph,
  $\SC_{\mathcal{A}}(G,\crp{E})$ is a subgroup.
  Pick a Clifford pair $(\lambda,\kappa)\in \mathcal{A}$
  with corresponding exact sequence
  \[ \begin{tikzcd}
       1 \rar & N \rar[hook]   & \oG \rar{\kappa} & G \rar & 1.
     \end{tikzcd}
  \]
  Let $K= \Ker \lambda$.
  Since $\lambda$ is assumed to be semi-invariant over
  $\crp{F}$, it follows that
  $K$ is normal in $\oG$.
  View $\lambda$ as character $\lambda_1$ of $N/K$
  and let $\kappa_1\colon \oG/K\to G$ be the map
  induced by $\kappa$.
  From the definitions, it follows that
  $\brcls{\lambda}{\kappa}{\crp{F}}
   = 
   \brcls{\lambda_1}{\kappa_1}{\crp{F}}$.
  Obviously, $(\lambda_1,\kappa_1)\in \mathcal{C}$.
  Thus $\SC_{\mathcal{A}}(G,\crp{E})
        \subseteq 
        \SC_{\mathcal{C}}(G,\crp{E})$.
  The reverse inclusion is trivial, and thus
  $\SC_{\mathcal{C}}(G,\crp{E})
    =\SC_{\mathcal{A}}(G,\crp{E})$
  is a subgroup.  
\end{proof}

\section{Dependence on the field}
\label{sec:depfield}
In this section, we prove Proposition~\ref{p:fieldindep}
from the introduction in the case where 
$Z=\crp{E}$ is a field.
However, let us first note that the uniqueness of the field 
$\crp{K}\subseteq \compl$ in Proposition~\ref{p:fieldindep}
follows in any case, because when 
$Z \iso \Z(\theta,\kappa,\crp{F})$, then 
$Z^G$ is 
a finite abelian Galois extension of $\crp{F}$
and thus isomorphic (as algebra over $\crp{F}$) to exactly one
field $\crp{K} \subseteq \compl$.

\begin{lemma}\label{l:brclsfielddep}
  Let $(\theta,\kappa)$ be a Clifford pair over $G$ that is 
  semi-invariant over the field\/ $\crp{F}$.
  Let\/ $\crp{K}$ be a field such that\/
  $ \crp{F} \subseteq \crp{K}
            \subseteq \crp{F}(\theta)^G $.
  Then 
  $\brcls{\theta}{\kappa}{\crp{F}}
  = \brcls{\theta}{\kappa}{\crp{K}}$.
\end{lemma}
\begin{proof}
  Note that the condition on $\crp{K}$ implies that
  $\crp{F}(\theta) = \crp{K}(\theta)$, and thus both
  $\brcls{\theta}{\kappa}{\crp{F}}$ and 
  $\brcls{\theta}{\kappa}{\crp{K}}$ are elements of the same
  Brauer-Clifford group $\BrCliff(G,\crp{F}(\theta))$.
  
  Let $V$ be a $\theta$-quasihomogeneous 
  module over $\crp{F}\oG$, 
  where, as usual, $\oG$ is the domain of $\kappa$
  and $N=\Ker\kappa$.  
  Since $\crp{K}$ is isomorphic to a subfield of
  $\Z(\crp{F}Ne_{\theta,\crp{F}})^G$,
  we can view $V$ as a $\crp{K}\oG$\nbd module.
  Then $\enmo_{\crp{K}N}(V)=\enmo_{\crp{F}N}(V)$,
  and the result follows.
\end{proof}
The next result contains Proposition~\ref{p:fieldindep} in the case 
where $Z = \crp{E}$ is a field.
\begin{prop}\label{p:fieldindep_fields}
  Let $G$ be a group which acts on the field $\crp{E}$, 
  and let 
  \[ \crp{F}\subseteq \crp{K}\subseteq\crp{E}^G
             \subseteq \crp{E} \subseteq \crp{F}(\eps)
  \]
  be a chain of fields, where 
  $\eps$ is a root of unity.
  Then 
  \[\SC^{(\crp{F})}(G,\crp{E}) = 
   \SC^{(\crp{K})}(G,\crp{E}).
  \]
\end{prop}
\begin{proof}
  The inclusion
   $\SC^{(\crp{F})}(G,\crp{E}) \subseteq 
     \SC^{(\crp{K})}(G,\crp{E})$
  is a direct consequence of 
  Lemma~\ref{l:brclsfielddep}:
  If $\brcls{\theta}{\kappa}{\crp{F}} \in \SC^{(\crp{F})}(G,\crp{E})$,
  then $\crp{E} = \crp{F}(\theta)$, and thus 
  $ \crp{K} \subseteq \crp{F}(\theta)^G $, 
  so Lemma~\ref{l:brclsfielddep} applies.
  
  Conversely, pick 
  $\brcls{\theta}{\kappa}{\crp{K}}
   \in \SC^{(\crp{K})}(G,\crp{E})$.
  Then $\crp{E}=\crp{K}(\theta)$, 
  but it may happen that $\crp{F}(\theta) < \crp{E}$
  (equivalently, $\crp{K}\not\subseteq \crp{F}(\theta)$).
  However, by Lemma~\ref{l:scid} there is a 
  Clifford pair $(\phi,\pi)$ such that
  \[ [\crp{E}]= \brcls{\phi}{\pi}{\crp{F}} 
              =\brcls{\phi}{\pi}{\crp{K}},
  \]
  where the second equality follows from Lemma~\ref{l:brclsfielddep}.
  We have $\crp{E}=\crp{F}(\phi) = \crp{F}(\theta \times \phi)$,
  and in particular 
  $\crp{K}\subseteq \crp{F}(\theta \times \phi)^G$.  
  Applying first Lemma~\ref{l:scprod}, then 
  Lemma~\ref{l:brclsfielddep} again, we get
  \begin{align*}
   \brcls{\theta}{\kappa}{\crp{K}}
      &= \brcls{\theta\times \phi}{\kappa \times_G \pi}{\crp{K}}
      \\
      &= \brcls{\theta\times \phi}{\kappa \times_G \pi}{\crp{F}}
     \in \SC^{(\crp{F})}(G,\crp{E}).\qedhere
  \end{align*}
\end{proof}
From the proof, we get a statement about restricted Schur-Clifford groups.
\begin{cor}
  In the situation of
  Proposition~\ref{p:fieldindep_fields},
  let $\mathcal{X}$ be a class of Clifford pairs.
  Then 
  \[\SC_{\mathcal{X}}^{(\crp{F})}(G,\crp{E}) \subseteq 
       \SC_{\mathcal{X}}^{(\crp{K})}(G,\crp{E}).
  \]
  If there is a Clifford pair $(\phi,\pi)$ such that
  $[\crp{E}]= \brcls{\phi}{\pi}{\crp{F}}$
  and such that
  $(\theta\times \phi, \kappa\times_G\phi)\in \mathcal{X}$
  for all $(\theta,\kappa)\in \mathcal{X}$,
  then 
  \[\SC_{\mathcal{X}}^{(\crp{F})}(G,\crp{E}) = 
         \SC_{\mathcal{X}}^{(\crp{K})}(G,\crp{E}).
  \]
\end{cor}
As an example consider the class $\mathcal{A}$ of
Clifford pairs $(\theta,\kappa)$ such that 
$N=\Ker \kappa$ is abelian.
If there is a root of unity $\zeta$ such that
$\crp{E}=\crp{F}(\zeta)$,
then 
$\SC_{\mathcal{A}}^{(\crp{F})}(G,\crp{E})
 = 
 \SC_{\mathcal{A}}^{(\crp{K})}(G,\crp{E})$.
On the other hand, it may happen that
$\crp{E}=\crp{K}(\zeta)$ for some root of unity
$\zeta$, but
$\crp{E}$ over $\crp{F}$ is not cyclotomic.
(Example: $\crp{F}=\rats$, $\crp{K}=\rats(\sqrt{2})$
 and $\crp{E}= \rats(\sqrt{2}, \zeta)$
 where $\zeta$ is a primitive third root of unity.)
Then 
$\SC_{\mathcal{A}}^{(\crp{F})}(G,\crp{E})$
is empty, while
$\SC_{\mathcal{A}}^{(\crp{K})}(G,\crp{E})$
is a subgroup of $\BrCliff(G,\crp{E})$
(see the proof of Corollary~\ref{c:cyclickernelclass}).

\section{Restriction}
\label{sec:res}
Let $G$ and $H$ be finite groups and
$\eps\colon H\to G$ a group homomorphism.
Let $R$ be a $G$\nbd ring.
Then $R$ can be viewed as $H$-ring and
$\eps$ induces a group homomorphism
$\BrCliff(G,R)\to \BrCliff(H,R)$.
Let $Z$ be a commutative simple $G$\nbd algebra
over $\crp{F}$.
Recall that this means that $Z$ has no nontrivial 
$G$\nbd invariant ideal,  
and that such a $Z$ is a direct product of fields 
which are permuted transitively by $G$~\cite[Proposition~2.12]{turull11}.
Then $Z$ may not be simple as an $H$\nbd algebra
(with the notable exception of the case where 
 $Z$ is a field).
But if $e$ is a primitive idempotent in $Z^H=\C_Z(H)$,
then $Ze$ is simple as $H$\nbd algebra.
The map $Z\to Ze$ is a homomorphism of $H$\nbd algebras
and induces a homomorphism of Brauer-Clifford groups
$\BrCliff(H,Z)\to \BrCliff(H,Ze)$.
\begin{prop}\label{p:restriction}
  Assume the situation just described.
  Then the induced homomorphism
  $\BrCliff(G,Z)\to \BrCliff(H,Ze)$ maps
  $\SC(G,Z)$ into $\SC(H,Ze)$.  
\end{prop}
\begin{proof}
  Let $(\theta,\kappa)$ be a Clifford pair
  over $G$.
  Consider
  \[ \begin{tikzcd}
       1 \rar & N \rar \dar[equals] 
              & \oH\rar{\rho} \dar{\pi} 
              & H \rar \dar{\eps}
              & 1 
              \\
       1 \rar & N \rar
              & \oG \rar{\kappa}
              & G \rar
              & 1 \rlap{,}
     \end{tikzcd}
  \]
  where $\oH$ is the pull-back
  of the morphisms $\kappa\colon \oG\to G$ and
  $\eps\colon H\to G$.
  Then $(\theta,\rho)$ is a Clifford pair over
  $H$. 
  (In the special case where $\eps\colon H\to G$ is an inclusion
   of a subgroup, 
   the group $\oH$ is simply $\kappa^{-1}(H)$
   and $\rho$ the restriction of $\kappa$
   to $\oH$.)
  The idempotent
  $(e_{\theta,\crp{F}})^H$ is a sum of $H$\nbd conjugates
  of $e_{\theta,\crp{F}}$ and thus
  $(e_{\theta,\crp{F}})^H(e_{\theta,\crp{F}})^G
    =(e_{\theta,\crp{F}})^H$.
  It follows that 
  $\Z(\theta,\rho,\crp{F})$
  is a direct summand of
  $\Z(\theta,\kappa,\crp{F})$.
  Conversely, any direct summand of
  $\Z(\theta,\kappa,\crp{F})$ as $H$\nbd algebra
  is of the form
  $\Z(\theta^g,\rho,\crp{F})$ 
  for some $g\in G$.
  
  Let $V$ be a $\theta$\nbd quasihomogeneous module
  over $\crp{F}\oG$
  such that $S=\enmo_{\crp{F}N}(V)$
  is a $G$\nbd algebra in 
  $\brcls{\theta}{\kappa}{\crp{F}}
   \in \SC(G,Z)$, and let $e\in Z^H$
   be a primitive idempotent.
  Then $Ve$ becomes a
  module over $\crp{F}\widehat{H}$ via
  $\pi\colon \widehat{H}\to \widehat{G}$.
  Thus $\enmo_{\crp{F}N}(Ve) \iso \enmo_{\crp{F}N}(V)e=Se$ is an 
  $H$\nbd algebra
  in $\brcls{\theta^g}{\rho}{\crp{F}}$
  for some $g\in G$.
  Since the map
  $\BrCliff(G,Z)\to \BrCliff(H,Ze)$
  sends the equivalence class of $S$ to the 
  equivalence class of the $H$\nbd algebra $Se$,
  this proves the result.  
\end{proof}

\section{The Schur-Clifford group over simple 
         \texorpdfstring{$G$}{G}-rings that are not a field}
\label{sec:notfield}
Let $Z$ be a simple $G$-ring.
Then $Z$ is isomorphic to a direct product of fields
which are permuted transitively 
by the group $G$~\cite[Proposition~2.12]{turull11}.
More precisely, let $e\in Z$ be a primitive idempotent.
Then $\crp{E}=Ze$ is a field.
Let $H=G_e=\{g\in G\mid e^g=e\}$ be the stabilizer of $e$.
Then $e^ge=0$ for all $g\in G\setminus H$.
Choose $T\subseteq G$ with 
$G=\bigdcup_{t\in T}Ht$.
Then
$ 1 = \sum_{t\in T} e^t$ is a primitive idempotent decomposition
of $1$ in $Z$.
Thus
\[ Z = \bigtimes_{t\in T} Ze^t 
     \iso \bigtimes_{t\in T} \crp{E}.
\]
The group $H$ acts on the field $\crp{E}=Ze$.
We recall the following 
result of Turull~\cite[Theorem~5.3]{turull09}:
\begin{prop}\label{p:cliffordri}
  With the notation just introduced, 
  the map sending a $G$\nbd algebra $S$ over $Z$ to 
  the $H$\nbd algebra $Se$ over $\crp{E}=Ze$
  defines an isomorphism
  \[ \BrCliff(G,Z)\iso \BrCliff(H, \crp{E}).
  \]  
\end{prop}
The following is Theorem~\ref{t:reductofields} from the introduction:
\begin{thm}\label{t:sccliffcor}
  The isomorphism of Proposition~\ref{p:cliffordri}
  restricts to an isomorphism
  \[ \SC(G,Z)\iso \SC(H,\crp{E}).
  \]  
\end{thm}
(As mentioned earlier, this theorem also completes the proofs of
Theorem~\ref{t:main} and Proposition~\ref{p:fieldindep} 
in the general case.)

It follows from Proposition~\ref{p:restriction}
that $\SC(G,Z)$ is mapped into
$\SC(H,\crp{E})$.
The nontrivial part is to show that
the restriction
$\SC(G,Z)\to \SC(H,\crp{E})$
is onto.
To do that, we need some facts about wreath products
and coset action which we recall now.
Choose a right transversal $T$ of $H$ in $G$,
so that $G=\bigdcup_{t\in T}Ht$. 
    The action of $G$ on the right cosets $Hg$ defines an action
    of $G$ on $T$, which we denote by $t\circ g$.
    Thus for $t\in T$ and $g\in G$,
    the element $t\circ g$ is the unique element in
    $ Htg\cap T$.
    We may write
    $tg= h(t,g) (t\circ g) $ with $h(t,g)\in H$.    
    The map $t\mapsto t\circ g$ is
    a permutation $\sigma=\sigma(g)\in \sym{T}$, 
    where $\sym{T}$ denotes the group of permutations of $T$.
    
    The symmetric group $\sym{T}$ acts on
    $H^T$, the set of maps $T\to H$, by
    $\big((h_t)_{t\in T}\big)^{\sigma}
     = (h_{t\sigma^{-1}})_{t\in T}$.
    The semidirect product of $\sym{T}$ and 
    $H^T$ is also known as the wreath product of $\sym{T}$ and
    $H$ and denoted by $H \wr\sym{T}$.  
    For every $g\in G$, set
    \[ g\phi = (h(t,g))_{t\in T} \sigma(g) 
             = \sigma(g) (h(t\circ g^{-1}, g))_{t\in T}
             \in H\wr \sym{T}.
    \]              
    It turns out that the map
    $\phi\colon G\to H\wr \sym{T}$
    is a group homomorphism~\cite[Lemma~13.3]{CRMRT1}.
    For later reference, we summarize:
\begin{lemma}\label{l:wreath_hom}
      Let $G$ be a group, $H\leq G$ and
      suppose $G=\bigdcup_{t\in T}Ht$.
      Then there is a group homomorphism
      $\phi\colon G \to H\wr \sym{T}$ such that
      \[\sigma (h_t)_{t\in T}= g\phi
            \iff 
            Ht\sigma = Htg
            \text{ and }
            tg= h_{t\sigma}t\sigma
            \text{ for all } t\in T.
          \]
\end{lemma}
    One may identify $T$ with $\Omega=\{Hg\mid g\in G\}$
    and $\sym{T}$ with $\sym{\Omega}$,
    but even then $\phi$ does depend on the choice
    of the right transversal $T$.

    Suppose that
    $\kappa\colon \oH\to H$ 
    is a surjective homomorphism with kernel $M$.
    We want to construct a group $\oG$ and a surjective homomorphism
    $\oG\to G$, where $H\leq G$ as before.
    The homomorphism $\kappa$ yields, in an obvious way, 
    a surjective homomorphism
    $\oH\wr\sym{T}\to H\wr\sym{T}$ with
    kernel $M^T$.
    We write $\kappa\wr 1$ for this homomorphism.
    We define $\oG$ to be the pullback of
    \[ \begin{tikzcd}
           {} &G \dar{\phi}\\
           \oH\wr\sym{T} \rar{\kappa\wr 1} &H\wr \sym{T}\rlap{,}
       \end{tikzcd}
    \]
    that is, we set
    \[ \oG= \{(\sigma(h_t)_{t\in T}, g)\in (\oH\wr\sym{T})\times G
              \mid
              g\phi = \sigma (h_t \kappa)_{t\in T}
            \}.
    \]
    Since $\kappa\wr 1$ is surjective,
    the canonical homomorphism
    $\oG\to G$ is surjective and has kernel
    canonically isomorphic with 
    $ M^T$.
    Thus:
\begin{lemma}\label{l:exseti}
      Let 
      \[ \begin{tikzcd}
           1 \rar & M \rar  & \oH \rar{\kappa} & H  \rar &1
         \end{tikzcd}
      \]
      be an exact sequence of groups and suppose
      $H\leq G$.
      Let $T$ and $\phi\colon G\to H\wr \sym{T}$
      be as in Lemma~\ref{l:wreath_hom}.
      Then there exists a group
      $\oG$ and a homomorphism
      $\kappa^{\tensor G}\colon \oG \to G$      
      such that the following diagram is commutative 
      and has exact rows:
    \[\begin{tikzcd}
        1 \rar & M^T \dar[xshift=-0.3em,equals] \rar 
               & \oG \dar \rar{\kappa^{\tensor G}} 
               & G \dar{\phi} \rar
               & 1 \\
        1 \rar & M^T \rar 
               & \oH \wr \sym{T} \rar{\kappa\wr 1}
               & H \wr \sym{T} \rar
               & 1.             
      \end{tikzcd}
    \]    
\end{lemma}    
     The construction depends on the choice 
     of the transversal $T$.
     It can be shown that the isomorphism type
     of the extension in the first row is independent
     of $T$.
     Maybe this justifies the notation $\kappa^{\tensor G}$,
     but otherwise we will not need this.
   
  \hypertarget{defi:conjugatemodule}{%
  We need another piece of notation.}
  Let $Z$ be a $G$\nbd ring and $W$ a
  $Z$\nbd module, and fix some $g\in G$.  
  We call a $Z$\nbd module $X$ a 
  \defemph{$g$\nbd conjugate}
  \label{defi:conjugatemodule}
  of $W$, if there is an isomorphism of 
  abelian groups $\alpha\colon W\to X$ such that
  $(wz)\alpha= w\alpha z^g$ for all $w\in W$ and
  $z\in Z$. 
  The map $\alpha$ is called a $g$\nbd isomorphism.
  Such a conjugate always exists, for example
  we can take $X=W$ as abelian group and define a new multiplication
  of $Z$ on $W$ by $w\bullet z= w z^{g^{-1}}$.
  Or we could take $X= W\tensor g\subseteq W\tensor_{Z}Z[G]$,
  where $Z[G]$ denotes the skew group ring.
  It is not difficult to show that all
  $g$\nbd conjugates of a given module $W$ are isomorphic.
  (This terminology is due to Riehm~\cite{Riehm70}.)

\begin{proof}[Proof of Theorem~\ref{t:sccliffcor}]
  That $\SC(G,Z)$ is mapped into
  $\SC(H,\crp{E})$
  follows from Proposition~\ref{p:restriction},
  as we mentioned already.
  
  Conversely, assume that 
  $(\theta, \kappa)$ is a Clifford pair over the group $H$
  such that 
  $\Z(\theta,\kappa,\crp{F})\iso \crp{E}$.
  Let $M$ be the kernel of $\kappa$ and
  $\oH$ its domain.
  Let $W$ be a $\theta$\nbd quasihomogeneous 
  $\crp{F}\oH$\nbd module
  and $A=\enmo_{\crp{F}M}(W)$.
  Thus $\brcls{\theta}{\kappa}{\crp{F}}\in \SC(H,\crp{E})$ 
  is the equivalence class of $A$.
  To prove Theorem~\ref{t:sccliffcor}, we have to find a Clifford pair
  over $G$ such that the isomorphism of Proposition~\ref{p:cliffordri} 
  maps its Brauer-Clifford class to 
  $\brcls{\theta}{\kappa}{\crp{F}}$.
  
  In the following, we use the notation of Lemma~\ref{l:exseti}.
  In particular, let $\kappa^{\tensor G}\colon \oG\to G$ be 
  the homomorphism of Lemma~\ref{l:exseti}
  with kernel $N:=M^T$.
  We will show that
  the isomorphism of Proposition~\ref{p:cliffordri}
  maps
  \[ \brcls{\theta\times \underbrace{1_M\times\dotsm\times 1_M}_{\abs{T}-1}
        }{\kappa^{\tensor G}
        }{\crp{F}} \in \SC(G,Z)
  \]
  to $\brcls{\theta}{\kappa}{\crp{F}}$,
  if $\theta\neq 1_M$. 
  The case $\theta=1_M$ will be treated separately.
  
  First we define an
  $\crp{F}\oG$\nbd module $V$.
  Recall that $G=\bigdcup_{t\in T}Ht$.
  Since $\crp{E}\iso \Z(\theta,\kappa,\crp{F})$,
  we may view $W$ as an $\crp{E}$\nbd module.
  Via the homomorphism $Z\to\crp{E}$, we can view
  $W$ as a $Z$\nbd module.
  Since $G$ acts on $Z$,
  it makes sense to speak of $g$\nbd conjugates for $g\in G$.
  We now choose a $t$\nbd conjugate $W_t$ and a
  $t$\nbd isomorphism $\gamma_t\colon W\to W_t$
  for every $t\in T$. 
  Every $\gamma_t$ is an isomorphism of 
  vector spaces over $\crp{F}\subseteq \crp{E}^H$.
  Let $V$ be the direct sum of the $W_t$'s:
  \[ V := \bigoplus_{t\in T} W_t.
  \]
  Every element $v\in V$ can be written uniquely as
  $v=\sum_{t\in T} w_t\gamma_t$ with $w_t\in W$.
  The wreath product $\oH \wr \sym{T}$ acts on $V$ by
  \[ \left(\sum_{t\in T} w_t\gamma_t \right) 
      \big(\sigma \cdot (h_t)_{t\in T}\big)
     = \sum_{t\in T} w_{t\sigma^{-1}}h_t \gamma_t.
  \]
  Using the homomorphism $\oG\to \oH\wr\sym{T}$,
  we can view $V$ as $\crp{F}\oG$\nbd module.
  
  Let $S$ be a simple $G$\nbd algebra over $Z$ such that
  $Se\iso A = \enmo_{\crp{F}M} (W)$.
  Every element of $S$ can be written uniquely
  as $\sum_{t\in T} a_t^t$ with $a_t\in Se\iso A$.
  The algebra $S$ is determined up to isomorphism
  by this property.
  (As a module over the skew group ring $Z[G]$,
   the algebra $S$ is isomorphic to $A\tensor_{Z[H]}Z[G]$.)
  The algebra $S$ acts on $V$ by
  \[ \left(\sum_{t\in T}w_t\gamma_t\right)
     \left(\sum_{t\in T} a_t^t\right)
     = \sum_{t\in T}w_ta_t\gamma_t.
  \]
  It is routine to verify that this action yields
  an $G$\nbd algebra homomorphism
  $S\to \enmo_{\crp{F}N}(V)$,
  and this homomorphism is injective.
  Next we want to show that it is onto.
  Here we will need to assume that $\theta\neq 1$.
  
  Let $\beta\in \enmo_{\crp{F}N}(V)$.
  For $t$, $u\in T$ and $w\in W$, there are elements
  $w\beta_{u,t}\in W$ such that 
  \[ w\gamma_u \beta =
     \sum_{t\in T} w\beta_{u,t}\gamma_t.
  \]
  This defines maps $\beta_{u,t}\colon W\to W$
  for each $u$, $t\in T$.
  Let $(m_t)_{t\in T}\in N=M^T$ be arbitrary.
  From
  \[ w\gamma_u\beta(m_t)_{t\in T} 
      = w\gamma_u(m_t)_{t\in T} \beta
      = w m_u \gamma_u \beta \]
  it follows that
  $w\beta_{u,t}m_t = wm_u \beta_{u,t}$.
  For $u=t$ this yields
  $\beta_{u,u}\in \enmo_{\crp{F}M}(W)=A$.
  For $u\neq t$ this yields
  $wm\beta_{u,t} = w\beta_{u,t}=w\beta_{u,t}m'$ for
  all $m$, $m'\in M$.
  We know that $W\iso tU$ for some
  simple $\crp{F}M$\nbd module $U$,
  and $U$ is not the trivial $\crp{F}N$\nbd module, 
  since we assume $\theta \neq 1_M$.
  Thus no non-zero element of $U$ is fixed by all of $M$.
  It follows that $\beta_{u,t}=0$ for $u\neq t$.
  Thus
  $\beta= \sum_{t\in T} (\beta_{t,t})^t\in S$.
  We have shown that $S\iso \enmo_{\crp{F} N}(V)$.
  
  Since $S\iso \enmo_{\crp{F}N}(V)$ is a simple $G$\nbd algebra,
  it follows that $V$ is quasihomogeneous
  over some character in $ \Irr M^T$.
  Now $W_1=W$ is a direct summand of $V$ as 
  $\crp{F}N$\nbd module, 
  and the action of $(m_t)_{t\in T}$
  on this summand is given by $w(m_t)_t = wm_1$.
  Since $\theta$ is a constituent of the character of $W_{\crp{F}M}$,
  it follows that $\theta\times 1_M \times \dotsm \times 1_M$
  is a constituent of the character of $V$.
  Thus $S$ is a simple $G$\nbd algebra in
  $\brcls{\theta\times 1_M \times \dotsm \times 1_M}{
          \kappa^{\otimes G}
         }{ \crp{F} }$ as claimed.
  
  In the remaining case where $\theta=1_M$ is the trivial character, 
  it follows that
  $M$ acts trivially on $W$ and we may view $W$ as an
  $\crp{F}H$\nbd module.
  Thus $A\iso \mat_n(\crp{F})$ for some positive integer 
  $n$
  and $\crp{F}=\crp{E}$.
  Then $A$ is equivalent to the trivial
  $H$\nbd algebra $\crp{E}$.
  We need to show that 
  $Z=\bigtimes_{t\in T} \crp{E}$
  is a Schur $G$\nbd algebra, where
  $G$ acts simply by permuting the factors.
  
  For this, let $M$ be a finite group such that there is 
  a nontrivial, absolutely simple 
  $\crp{E}M$\nbd module $U$.
  (For example, take a cyclic group of order $2$ for $M$.)
  Let $\oG=M\wr G$ be the wreath product with respect
  to the action of $G$ on $T$, 
  and $N=M^T\nteq \oG$.
  Then $\oG$ acts on $V=U^T$ by
  \[(u_t)_{t\in T} \sigma(m_t)_{t\in T}
    =(u_{t\sigma^{-1}}m_t)_{t\in T},\]
  and we have $\enmo_{\crp{E}N}(V)\iso Z$
  as desired.
\end{proof}

\section{Corestriction}
\label{sec:cores}
Let $G$ be a finite group and $H\leq G$.
  Let $\crp{F}$ be a field of characteristic zero
  and $Z$ a commutative $G$\nbd algebra over
  $\crp{F}$, so that $\crp{F}\leq Z^G$.
  Given an $H$\nbd algebra $A$ over $Z$,
  we showed in~\cite{ladisch15a} 
  how to construct a certain 
  $G$\nbd algebra $\Cores_H^G(A)$ over $Z$.
  This defines a group homomorphism
  $\Cores_H^G \colon \BrCliff(H,Z)\to \BrCliff(G,Z)$
  called \defemph{corestriction}.
  If $Z$ is simple as $H$\nbd algebra, then
  it is also simple as $G$\nbd algebra,
  and we may ask if the corestriction map
  maps
  $\SC^{(\crp{F})}(H,Z)$ into $\SC^{(\crp{F})}(G,Z)$.
  It is important to fix the field $\crp{F}$ here.
  For example, assume that $Z=\crp{E}$ is a field.
  Then it may happen that the Galois extension
  $\crp{E}/\crp{E}^H$ is abelian, but the
  Galois extension $\crp{E}/\crp{F}$
  (with $\crp{F}=\crp{E}^G$, say) is not.  
  Then $\SC(H,\crp{E})=\SC^{(\crp{E}^H)}(H,\crp{E})$
  is not empty by Lemma~\ref{l:scid} 
  and Kronecker-Weber, while
  $\SC(G,\crp{E})=\SC^{(\crp{F})}(G,\crp{E})$
  and $\SC^{(\crp{F})}(H,\crp{E})$ are empty.  
  The next result has content only when $\SC^{(\crp{F})}(H,Z)$
  is non-empty.
\begin{thm}\label{t:coresschurcl}
    With the notation just introduced, 
    the corestriction homomorphism
    maps $\SC^{(\crp{F})}(H,Z)$ into $\SC^{(\crp{F})}(G,Z)$.
\end{thm}
\begin{proof}
    Let $(\theta,\kappa)$ be a Clifford pair,
    where $\theta\in \Irr M$ with
    $Z\iso \Z(\theta,\kappa,\crp{F})$ as $H$\nbd algebras.
    Suppose that $W$ is an 
    $\crp{F}\oH$\nbd module which is 
    $\theta$\nbd quasihomogeneous,
    and let $S=\enmo_{\crp{F}M}(W)$.
    We will use the notation introduced in
    Lemmas~\ref{l:wreath_hom} and~\ref{l:exseti}.   
    In particular, let $G=\bigdcup_{t\in T}Ht$
    and recall the diagram
        \[\begin{tikzcd}
            1 \rar & M^T \dar[xshift=-0.3em,equals] \rar 
                   & \oG \dar \rar{\kappa^{\tensor G}} 
                   & G \dar{\phi} \rar
                   & 1 \\
            1 \rar & M^T \rar 
                   & \oH \wr \sym{T} \rar{\kappa\wr 1}
                   & H \wr \sym{T} \rar
                   & 1             
          \end{tikzcd}
        \]    
    with exact rows from Lemma~\ref{l:exseti}.     
    We are going to show that 
    $\Cores_H^G(S)=S^{\tensor G}\iso \enmo_{\crp{F}[M^T]}(V)$
    for some 
    $\crp{F}\oG$\nbd module $V$.
    
      Since $\enmo_{\crp{F}M}(W)=S$, 
      we can view $W$ as a module over
      $Z= \Z(S)$.
      Recall that $G$ acts on $Z$.    
      We now choose a $t$\nbd conjugate $W_t$ 
      (\hyperlink{defi:conjugatemodule}{see}
       p.~\pageref{defi:conjugatemodule}) 
      and a
      $t$\nbd isomorphism $\gamma_t\colon W\to W_t$
      for every $t\in T$.
      Every $\gamma_t$ is an isomorphism of 
      vector spaces over $\crp{F}\subseteq Z^H$.
      We set
      \[ V = \bigtensor_{t\in T} W_t
         \quad\text{(tensor product over $Z$)}.
      \]
      We want to define an action of $\oG$ on $V$.
      So let $\big(\sigma(h_t)_{t\in T}, g\big)\in \oG$ and
      let $w_t\in W$.
      Define
      \[ \left(\bigtensor_{t\in T} w_t\gamma_t \right)
         \big( \sigma (h_t)_{t\in T},g \big)
         = \bigtensor_{t\in T} w_{t\circ g^{-1}} h_t \gamma_t.
      \]
      First we show that this yields a well defined
      endomorphism of $V$.
      So let $w_t\in W$ and $z\in Z$, and fix $t_0\in T$.
      Suppose
      \[ u_t = \begin{cases}
                  w_t, \quad &\text{if } t\neq t_0,\\
                  w_t z^{t_0^{-1}}, \quad &\text{if } t = t_0. 
               \end{cases}
      \]
      Then $u_{t_0}\gamma_{t_0}= w_{t_0}\gamma_{t_0}z$
      and
      \[ \bigtensor_{t\in T} u_t \gamma_t 
         = \left(\bigtensor_{t\in T} w_t\gamma_t
           \right) z. 
      \]
      Applying the definition to the left side yields
      \begin{align*}
        \left( \bigtensor_{t\in T} u_t\gamma_t \right)
          \big( \sigma (h_t)_{t\in T}, g \big)
        &= \bigtensor_{t\in T} u_{t\circ g^{-1}}h_t \gamma_t
         = \bigtensor_{t\in T} v_t\gamma_t,  
      \end{align*}
      where
      \[ v_t = \begin{cases}
                  w_{t\circ g^{-1}} h_t \quad 
                      &\text{if }t\neq t_0 \circ g, \\
                  w_{t_0}z^{t_0^{-1}}h_{t_0\circ g} \quad
                      &\text{if } t=t_0\circ g. 
              \end{cases}
      \]
      Thus
      \begin{align*}
        v_{t_0\circ g}\,\gamma_{t_0\circ g}
           = w_{t_0}\, z^{t_0^{-1}} h_{t_0\circ g}\, \gamma_{t_0\circ g}
           &= w_{t_0}  h_{t_0\circ g}\, \gamma_{t_0\circ g} \,
              z^{t_0^{-1} (h_{t_0\circ g}\kappa)\, t_0\circ g}
           \\
           &= w_{t_0} h_{t_0\circ g} \, \gamma_{t_0\circ g}\, z^{g},
      \end{align*}
      where the last equation follows from
      $(h_{t_0\circ g}\kappa)t_0\circ g = t_0 g$,
      which follows from the fact that
      $ \sigma (h_t\kappa)_{t\in T}= g\phi$
      (see Lemma~\ref{l:wreath_hom}).
      It follows that
      \[ \bigtensor_{t\in T} v_t\gamma_t
         = \left( \bigtensor_{t\in T} w_{t\circ g^{-1}}h_t\gamma_t
           \right) z^{g}.
      \]
      This does not depend on $t_0$.
      Thus multiplication with
      $\big( \sigma (h_t)_{t\in T}, g \big)$ yields a well-defined 
      endomorphism of $V$.
      
      Having established this, 
      it is routine to check that we actually have an action
      of $\oG$ on $V$.
      
      Our next goal is to show that
      $\enmo_{\crp{F}N}V\iso S^{\tensor G}$
      as $G$\nbd algebra over $Z$.
      First recall that
      $S^{\tensor G}\iso \bigtensor_{t\in T} S^{\tensor t}$,
      where the tensor product is over $Z$ and each
      $S^{\tensor t}$ is a $t$\nbd conjugate of $S$.
      We define an action of $S^{\tensor G}$ on the module
      $V=\bigtensor_{t\in T}W_t$ by
      \[ \left(\bigtensor_{t\in T} w_t\gamma_t \right) 
         \left(\bigtensor_{t\in T} s_t^{\tensor t} \right)
         = \bigtensor w_ts_t\gamma_t.
      \]
      It is routine to see that this yields a well-defined
      homomorphism 
      $S^{\tensor G}\to \enmo_{\crp{F}N}V$
      of $Z$\nbd algebras.
      We claim that it is actually an isomorphism and commutes
      with the action of $G$.
      
      The homomorphism is injective, since $S^{\tensor G}$
      is simple.
      We know that $S\iso \enmo_{\crp{F}M}W=\enmo_{Z M}W$.
      Clearly $Z\subseteq \enmo_{\crp{F}N}V$, so that
      $\enmo_{\crp{F}N}V=\enmo_{ZN}V$.
      But $ZN\iso ZM \tensor_Z \dotsm \tensor_Z ZM= (ZM)^{\tensor T}$
      (tensor product over $Z$ of $\abs{T}$ factors), and so
      \begin{align*}
        \enmo_{ZN}V &= \C_{\enmo_Z(V)}(ZN)
        \\            
                    &= \C_{\enmo_Z(V)}(ZM\tensor_Z \dotsm \tensor_Z ZM)
                    \\
                    &\iso S\tensor_Z \dotsm \tensor_Z S
                     \iso S^{\tensor G} \quad \text{(as $Z$-algebra)}.
      \end{align*}
      This shows that the homomorphism above is an isomorphism
      of $Z$\nbd algebras.
      
      Finally we have to show that this isomorphism is compatible
      with the action of $G$ on $S^{\tensor G}$ and
      $\enmo_{\crp{F}N}V$.
      Let $g\in G$ and 
      $\bigtensor_{t\in T} s_t^{\tensor t}\in S^{\tensor G}$.
      Then
      \begin{align*}
        \left(\bigtensor_{t\in T} s_t^{\tensor t}\right)^g
         &= \bigtensor_{t\in T} s_{t}^{\tensor tg}
          = \bigtensor_{t\in T} 
               \left( s_{t\circ g^{-1}}^{ (t\circ g^{-1})gt^{-1}
                                        }
               \right)^{\tensor t}.
      \end{align*}
      Thus 
      \[ \left(\bigtensor_{t\in T} w_t\gamma_t\right)
         \left( \bigtensor_{t\in T} s_t^{\tensor t} \right)^g
         = \bigtensor_{t\in T} w_t 
                 s_{t\circ g^{-1}}^{(t\circ g^{-1})gt^{-1}}\gamma_t.
      \]
      Now choose $\sigma\in \sym{T}$ and
      $h_t\in \oH$ ($t\in T$) such that
      $(\sigma(h_t)_{t\in T}, g)\in \oG$.
      Recall that this means that $t\sigma=t\circ g$ and
      $h_t \kappa = (t\circ g^{-1})gt^{-1}$.
      We compute
      \begin{align*}
      \MoveEqLeft[5] \left(\bigtensor_{t\in T} w_t\gamma_t
                  \right)
                  \big(\sigma(h_t)_{t\in T},g 
                  \big)^{-1}
                  \left( \bigtensor_{t\in T} s_t^{\tensor t} 
                  \right)
                  \big( \sigma (h_t)_{t\in T}, g
                  \big)
      \\
        &= \left( \bigtensor_{t\in T} w_{t\sigma}
                        h_{t\sigma}^{-1}
                        s_t \gamma_t
           \right)
           \big( \sigma(h_t)_{t\in T}, g\big)
        \\
        &= \bigtensor_{t\in T} w_th_t^{-1} 
                          s_{t\sigma^{-1}}
                          h_t \gamma_t
        \\
        &= \bigtensor_{t\in T} w_t 
             s_{t\circ g^{-1}}^{h_t\kappa} 
             \gamma_t
        \\
        &= \bigtensor_{t\in T} w_t 
                 s_{t\circ g^{-1}}^{(t\circ g^{-1})gt^{-1}}
                 \gamma_t.
      \end{align*}
      This finishes the proof.
    \end{proof}
    \begin{cor}\label{c:coreschars}
      Let $Z$ be a commutative $G$\nbd algebra and $H\leq G$.
      Assume that $Z$ has no nontrivial $H$\nbd invariant ideals.
      Let
      \[ 
         \begin{tikzcd}
            1 \rar & M \rar & \oH \rar{\kappa} & H \rar & 1
         \end{tikzcd}
      \]
      be an exact sequence of groups and $\theta\in \Irr M$
      with $Z\iso \Z(\theta,\kappa,\crp{F})$
      as $H$\nbd algebras.
      Write $B=\{\phi\in \Irr M \mid \phi_{|Z}\neq 0\}.$
      Then
      \[ \Cores \brcls{\theta}{\kappa}{\crp{F}}
         = \brcls{\bigtimes_{t\in T}\theta_t}{\kappa^{\tensor G}}{\crp{F}},
      \]          
      where each $\theta_t\in B$.
    \end{cor}
    \begin{proof}
      Note that $B$ is the orbit of $\theta$ under the action
      of $\Gamma \times H$, where
      $\Gamma = \Gal( \crp{F}(\theta)/\crp{F})$.
      Let $\beta= \sum_{\phi\in B} \phi$.
      Let $W$ be a $\theta$\nbd quasihomogeneous module.
      Then the character of $W_{M}$ 
      is an integer multiple of $\beta$.
      View $W^{\tensor \abs{T}}=\bigtensor_{t\in T}W$
      (tensor product over $\crp{F}$)
      as module over $\crp{F}[M^T]$.
      This module affords the character
      $\bigtimes_{t\in T} \beta$.
      Let 
      $V$ be the $\crp{F} \oG$\nbd module 
      constructed in the proof of Theorem~\ref{t:coresschurcl}.
      Consider the map
      \[ W^{\tensor \abs{T}} \ni
         \bigtensor_{t\in T} w_t
         \mapsto \bigtensor_{t\in T}w_t\gamma_t\in V.
      \]
      This maps a tensor product over $\crp{F}$ to a tensor product
      over $Z$.
      It is easy to see that this map
      is a homomorphism of $\crp{F}[M^T]$\nbd modules.
      Thus $V_N$ affords a character which is a sum of
      characters of the form $\bigtimes_{t\in T} \theta_t$
      with $\theta_t\in B$.
      This is the claim. 
    \end{proof}
    \begin{remark}
      The corollary does not tell the full story, 
      namely what the $\theta_t$ actually are. 
      Note that since $Z$ is assumed to be a $G$\nbd algebra,
      we have an action of $G$ on $B$.
      The restriction of this action to $H$ agrees with the
      usual action of $H\iso \oH/M$ on $B\subseteq \Irr M$.
      With somewhat more effort, one can show that
      the character of the module $V$ in the proof of 
      Theorem~\ref{t:coresschurcl}
      has the character $\bigtimes_{t\in T} \theta^{t^{-1}}$
      as constituent.
      Thus
      \[ \Cores \brcls{\theta}{\kappa}{\crp{F}}
               = \brcls{\bigtimes_{t\in T}\theta^{t^{-1}}
                       }{\kappa^{\tensor G}}{\crp{F}}.
      \] 
      But for the applications we have in mind, it is enough
      to know the weaker statement of Corollary~\ref{c:coreschars}.
    \end{remark}
  Recall that the Brauer-Clifford group is an abelian torsion group.
  It is thus the direct sum of its $p$\nbd torsion parts where $p$
  runs through the primes.
  The same is true for the Schur-Clifford group, if defined.
  Denote the $p$\nbd part of an abelian group $A$ by $A_p$.
  \begin{cor}
    Let $Z$ be a simple $G$\nbd algebra over $\crp{F}$,
    $P$ a Sylow $p$\nbd subgroup of $G$ for some prime $p$
    and assume that $Z$ is simple as $P$\nbd algebra.
    If $\SC^{(\crp{F})}(G,Z)$ is not the empty set,
    then 
    $\SC^{(\crp{F})}(G,Z)_p$ 
    is isomorphic to a subgroup of $ \SC^{(\crp{F})}(P,Z)_p$.
    Moreover, an equivalence class $a = [S]\in \BrCliff(G,Z)_p$ 
    belongs to
    $\SC^{(\crp{F})}(G,Z)$ if and only if\/
    $\Res^G_P(a)\in \SC^{(\crp{F})}(P,Z)$.    
  \end{cor}
  \begin{proof}
    By Theorem~6.3 in \cite{ladisch15a} 
    the composition
    \[ \begin{tikzcd}
          \BrCliff(G,Z) \rar{\Res^G_P} 
          & \BrCliff(P,Z) \rar{\Cores^G_P} 
          & \BrCliff(G,Z)          
       \end{tikzcd}
    \]
    maps an element $a=[S]$ to $a^{\abs{G:P}}=[S]^{\abs{G:P}}$.
    Thus $\Res^G_P$ is injective on the $p$\nbd torsion part.
    By Proposition~\ref{p:restriction},
    it maps
    $\SC^{(\crp{F})}(G,Z)_p$
    into $\SC^{(\crp{F})}(P,Z)_p$.
    
    Conversely, if $\Res^G_P(a)\in \SC^{(\crp{F})}(P,Z)$
    for some $a\in \BrCliff(G,Z)$, then
    $a^{\abs{G:P}}\in \SC^{(\crp{F})}(G,Z)$ by 
    Theorem~\ref{t:coresschurcl}.
    Thus if the order of $a$ is a power of $p$, then
    $a \in \erz{a^{\abs{G:P}}}$ is in $\SC^{(\crp{F})}(G,Z)$, too.
  \end{proof}
  
\section{A special case and a conjecture}
\label{sec:ausblick}
  The results of this paper suggest that one can carry through a study
  of the Schur-Clifford subgroup for various groups 
  $G$ and simple commutative $G$\nbd algebras $Z$.
  For example, we have seen that for $Z=\compl$ with trivial
  action of the group $G$, we have
  $\BrCliff(G,\compl) =\SC(G,\compl)$
  and that $\SC(G,\compl)\iso H^2(G,\compl^*)$ (Example~\ref{ex:complcohom}).
  Moreover, every element in $\SC(G,\compl)$ can be realized
  as the class of a Clifford pair $(\theta,\kappa)$ such that
  $\Ker \kappa $ is a cyclic central subgroup of $\oG$, the domain of
  $\kappa$.
  In a follow-up paper~\cite{ladisch14c}, we prove a similar result
  for arbitrary fields $\crp{E}$ with $G$\nbd action, 
  showing that every element in $\SC(G, \crp{E})$ 
  comes from a Clifford pair $(\theta,\kappa)$ 
  such that the normal subgroup $\Ker\kappa$
  (which is the domain of $\theta$)
  is cyclic by abelian.
  
  We might also discuss how $G$ and $Z$ bear upon group theoretic
  properties of $\SC(G,Z)$.
  We discuss now a slightly more general case than the case
  $\compl$.
  First, observe the following.
  Given an element of the Brauer-Clifford group
  $\BrCliff(G,\crp{E})$, where $\crp{E}$ is a field on which $G$
  acts, we may take its equivalence class in the Brauer group
  $\Br(\crp{E})$.
  This defines a group homomorphism
  $\BrCliff(G,\crp{E}) \to \Br(\crp{E})$.
  It is clear that this homomorphism maps the Schur-Clifford group
  into the Schur subgroup of the Brauer group.
  (In fact, this is the case $H=\{1\}$ of 
   Proposition~\ref{p:restriction}.)
  
  Turull has shown~\cite[Theorem~3.12]{turull09b} 
  that the kernel of the homomorphism is isomorphic to
  the second cohomology group $H^2 (G,\crp{E}^*)$.
  In the case where $G$ acts trivially on $\crp{E}=\crp{F}$,
  Turull has shown~\cite[Corollary~3.13]{turull09b}
  that $\BrCliff(G,\crp{F})\iso \Br(\crp{F})\times H^2 (G,\crp{F}^*)$.
  The Schur-Clifford group decomposes accordingly:  
  \begin{prop}
    Suppose that\/ $\crp{F}$ is a field on which the group $G$
    acts trivially.
    Let $\mathcal{S}(\crp{F})\leq \Br(\crp{F})$
    be the Schur subgroup of the Brauer group of\/ $\crp{F}$.
    Then 
    \[ \SC(G, \crp{F})
       \iso  \mathcal{S}(\crp{F})
          \times  A
       \quad\text{for some}\quad
       A \leq H^2(G, \crp{F}^*).
    \]
  \end{prop}
  \begin{proof}
    The homomorphism 
    $\beta\colon\BrCliff(G,\crp{F})\to \Br(\crp{F})$
    maps $\SC(G,\crp{F})$ into $\mathcal{S}(\crp{F})$.
    
    Conversely, there is a homomorphism
    $\Br(\crp{F}) \to \BrCliff(G,\crp{F})$
    which is induced by viewing a central simple $\crp{F}$\nbd algebra
    as a $G$\nbd algebra with trivial action of $G$.  
    This homomorphism shows that
    $\BrCliff(G,\crp{F})$ is a direct product of
    $\Br(\crp{F})$ and some other group.
    The other group is isomorphic to
     the second cohomology group $H^2(G, \crp{F}^*)$
     by the results of Turull mentioned before.
    
    To finish the proof, it remains to show that
    the homomorphism
    $\Br(\crp{F}) \to \BrCliff(G,\crp{F})$
    maps the Schur group $\mathcal{S}(\crp{F})$
    into the Schur-Clifford group
    $\SC(G,\crp{F})$.    
    Suppose that $\phi$ is the character of some group $N$
    with $\crp{F}(\phi)=\crp{F}$, so that
    $\phi$ defines an element $\llbracket \phi \rrbracket$ 
    of the Schur subgroup 
    $\mathcal{S}(\crp{F})$.
    More precisely, there is an $\crp{F}N$-module $V$
    whose character is a multiple of $\phi$,
    and then $\enmo_{\crp{F}N}(V)$ is a central simple
    $\crp{F}$\nbd algebra whose equivalence class
    defines an element $\llbracket \phi \rrbracket$.
    
    Let $\oG = N\times G$
    and $\kappa\colon \oG \to G$ be the projection on the second component.
    We identify $N$ with the kernel of $\kappa$.
    Then $(\phi, \kappa)$ is a Clifford pair over
    $G$ with
    $\beta \brcls{\phi}{\kappa}{\crp{F}} = 
     \llbracket \phi \rrbracket$.
    It is easy to see that the map
    $\llbracket \phi \rrbracket \mapsto
     \brcls{\phi}{\kappa}{\crp{F}}$ 
     is the restriction of the map
    $\Br(\crp{F})\to \BrCliff(G,\crp{F})$ 
     defined above.    
    The result follows.
  \end{proof}
  Given a Clifford pair $(\theta,\kappa)$ over $G$,
  on can construct directly a cocycle in $Z^2(G,\crp{F}^*)$,
  and then show that its cohomology class depends only on
  $\brcls{\theta}{\kappa}{\crp{F}}$.
  This yields an alternative proof of the proposition.
   
  The subgroup $A\leq H^2 (G,\crp{F})$ occurring in the last proposition
  has been studied by Dade~\cite{dade74}.
  He proved that if $m$ is the exponent of the group $A$,
  then $\crp{F}$ contains a primitive $m$\nbd th root of unity.
  For the Schur subgroup of the Brauer group, 
  the same result is true by a result of 
  Benard-Schacher~\cite[Proposition~6.2]{Yamada74}.
  Thus we have:
  \begin{cor}\label{c:scgroupexp}
    Suppose that $G$ acts trivially on the field\/ $\crp{F}$
    and let $m=\exp(\SC(G,\crp{F}))$.
    Then
    $\crp{F}$ contains a primitive $m$\nbd th root of unity.
  \end{cor}
  Schmid~\cite{schmid88} shows how to associate a 
  cohomology class to a semi-invariant Clifford pair 
  $(\theta,\kappa)$ over $\crp{F}$ 
  in a slightly more  general situation.
  (The assumption is that the Schur index of $\theta$ over $\crp{F}$
  and the index $\abs{G:G_{\theta}}$ of the inertia group of $\theta$ 
  in $G$ are coprime.)
  Moreover, Schmid shows that $\crp{F}(\theta)$ contains 
  a primitive $m$\nbd th root of unity if this 
  cohomology class has order $m$~\cite[Theorem~7.3]{schmid88}.
  We conclude this paper with the conjecture that 
  Corollary~\ref{c:scgroupexp} is true for all
  Schur-Clifford subgroups:
  \begin{conj}
    Let $\crp{E}\subseteq \compl$ be a field on which the group $G$ acts,
    and suppose $\SC(G,\crp{E})\neq \emptyset$.
    If $m$ is the exponent of $\SC(G,\crp{E})$,
    then $\crp{E}$ contains a primitive $m$\nbd th root of unity.
  \end{conj}

%
\printbibliography   
%

\end{document}